\documentclass[10pt,a4paper]{amsart}

\usepackage[english]{babel} 
\usepackage[T1]{fontenc} 
\usepackage{amsmath}
\usepackage{amssymb} 
\usepackage{ulem}
\usepackage{ucs} 
\usepackage[utf8x]{inputenc}
\usepackage[]{graphicx} 
\usepackage{latexsym}
\usepackage{amsthm} 
\usepackage{stmaryrd}
\usepackage{verbatim}
\usepackage{pgf,tikz}
\usetikzlibrary{arrows}

\usepackage{hyperref}

\newtheorem{defi}{Definition}[section] 
\newtheorem{theo}[defi]{Theorem}
\newtheorem{theorem}[defi]{Theorem}
\newtheorem{question}[defi]{Question}
\newtheorem{coro}[defi]{Corollary} 
\newtheorem{lemme}[defi]{Lemma}
\newtheorem{lemma}[defi]{Lemma}
\newtheorem{prop}[defi]{Proposition}
\newtheorem{conj}[defi]{Conjecture}

\theoremstyle{remark}
\newtheorem{example}[defi]{Example}
\newtheorem{rmk}[defi]{Remark}

\newcommand{\proj}{\mathbf{P}}

\newcommand{\R}{\mathbb{R}}

\newcommand{\C}{\mathbb{C}} 

\newcommand{\N}{\mathbb{N}}

\newcommand{\p}{\varphi}

\newcommand{\Cone}{\mathrm{Cone}}

\newcommand{\diag}{\mathrm{diag}}

\newcommand{\SL}{\mathrm{SL}}

\newcommand{\Log}{\mathrm{Log}}

\newcommand{\SO}{\mathrm{SO}}

\newcommand{\op}{\iota}

\newcommand{\norm}[1]{\Vert #1 \Vert}

\renewcommand{\Log}{\log}

\newcommand{\scal}[2]{\langle #1\, |\, #2 \rangle }
\newcommand{\tv}{\rightarrow}
\newcommand{\G}{\Gamma}
\newcommand{\g}{\gamma}

\DeclareMathOperator{\Card}{Card}
\DeclareMathOperator{\DimH}{DimH}

\DeclareMathOperator{\Sym}{Sym}

\newcommand{\fa}{\mathfrak{a}}
\newcommand{\fg}{\mathfrak{g}}
\newcommand{\fp}{\mathfrak{p}}
\newcommand{\fk}{\mathfrak{k}}

\newcommand{\cC}{\mathcal{C}}

\DeclareMathOperator{\Span}{Span}
\DeclareMathOperator{\ad}{ad}

\newcommand{\Hyp}{\mathbb{H}}

\renewcommand{\epsilon}{\varepsilon}
\renewcommand{\phi}{\varphi}

\title[Hausdorff dimension of limit sets for  Anosov representations]{Hausdorff dimension of limit sets for projective Anosov representations} 

\author{Olivier Glorieux}
\address{Olivier Glorieux: Lyc\'ee Chaptal, 45 Bd des Batignolles,  75008 Paris}

\author{Daniel Monclair}
\address{Daniel Monclair: Laboratoire de Math\'ematiques d'Orsay, Universit\'e Paris-Saclay, F-91405 Orsay Cedex, France}

\author{Nicolas Tholozan}
\address{Nicolas Tholozan: CNRS, \'ENS-PSL, 45 rue d'Ulm, 75005 Paris, France}

\begin{document}

\begin{abstract}
We study the relation between critical exponents and Hausdorff dimensions of limit sets for projective Anosov representations. We prove that the Hausdorff dimension of the symmetric limit set in $\proj(\R^{n}) \times \proj({\R^{n}}^*)$ is bounded between two critical exponents associated respectively to a highest weight and a simple root.
\end{abstract}

\maketitle

\tableofcontents

\section{Introduction}

\subsection{Critical exponents and Hausdorff dimension}

Let $\G$ be a discrete group of isometries of a metric space $(X,d)$. A well-known metric invariant of $\G$ is its \emph{critical exponent}, which measures the exponential growth rate of its orbits. It can be defined by
\[\delta_\Gamma = \limsup_{R\to +\infty} \frac{1}{R}\log\left(\Card\{g\in \G \mid d(x,g \cdot x) \leq R\}\right)~,\]
where $x$ is any base point in $X$.

When $(X,d)$ is the hyperbolic space $\Hyp^n$ and $\G$ is \emph{convex-cocompact} (i.e. acts cocompactly on a non-empty convex subset of $\Hyp^n$), Sullivan  \cite{sullivan1979density}  proved that the critical exponent of $\G$ equals the Hausdorff dimension of the \emph{limit set} of $\G$ inside $\partial_\infty \Hyp^n$. The proof relies on the Ahlfors regularity of the \emph{Patterson--Sullivan measure} on the limit set. 

This famous theorem has known a number of generalizations. See for instance \cite{DOP,roblin2003ergodicite,Coornaert} for generalizations to other discrete groups acting on hyperbolic spaces. This paper is mainly interested in extensions to other non-positively curved geometries. A fairly general version of Sullivan's theorem was given by Coornaert for a discrete group $\G$ acting convex-cocompactly on a Gromov hyperbolic space $X$ \cite[Corollaire 7.6]{Coornaert}. In this setting, the critical exponent equals the Hausdorff dimension of the limit set of $\G$ in $\partial_\infty X$ measured with respect to Gromov's ``quasi-distance'' on the boundary (see Section \ref{subsec-gromov metric on the boundary}). When $X$ is a Riemannian manifold with variable negative curvature, however, this metric typically differs from the visual metric (obtained by identifying the boundary at infinity with the tangent sphere at a basepoint). For instance, the visual metric on the boundary of the complex hyperbolic space $\Hyp^n_\C$ is the round metric on the unit sphere in $\C^n$, while the Gromov metric coincides with the \emph{Carnot--Caratheodory metric}.

There have also been several important works generalizing Patterson--Sullivan theory to discrete subgroups of a semi-simple Lie group $G$ of higher rank acting on its symmetric space $X$ \cite{link2005measures,quint2002mesures}. A new feature of the higher rank is the existence of several critical exponents corresponding to several $G$-invariant ``metrics'' on $X$. Quint studied in \cite{quint2002mesures} the dependence of those critical exponents on such a choice and constructed analogs of Patterson--Sullivan measures on the space $G/P_{min}$, where $P_{min}$ is a minimal parabolic subgroup (see Section \ref{subsec - Critical exponents symmetric space}).

The recently developed theory of Anosov subgroups of higher rank Lie groups motivates a further investigation of these generalizations. Anosov subgroups are in many aspects the ``right'' generalization of convex cocompact groups in rank $1$. In particular, they are Gromov hyperbolic, and their Gromov boundary is realized geometrically as a limit set in some flag variety $G/P$. It is natural to ask how the Hausdorff dimension of this limit set relates to the different critical exponents of the group.

An important feature of this new context is that the action of $G$ on $G/P$ is far from conformal as soon as the rank of $G$ is at least $2$. This makes Hausdorff dimensions of attractors hard to track (see \cite{PSW19} for further discussion).

\subsection{Main results}

The present work focuses on \emph{projective Anosov subgroups} of $\SL(n,\R)$. We explain in Section \ref{subsubsec-Anosov representation} that general Anosov subgroups of a semi-simple Lie group $G$ can be seen as projective Anosov groups after composing with a suitably chosen linear representation (see also \cite[Section 3.3]{gueritaud20017anosov}).

Let $\G$ be a projective Anosov subgroup of $\SL(n,\R)$. Then $\G$ is Gromov hyperbolic and comes with two injective equivariant maps $\xi : \partial_\infty \G \to \proj (\R^n)$ and $\xi^* : \partial_\infty \G \to \proj ({\R^n}^*)$. We denote by $\xi^{sym}$ the map $(\xi,\xi^*): \partial_\infty \G \to \proj (\R^n) \times \proj({\R^n}^*)$. If, moreover, $\G$ preserves a proper convex subset of $\proj(\R^n)$, then $\G$ is \emph{strongly projectively convex-cocompact} in the sense of \cite{DGK}. 

Given $g\in \SL(n,\R)$, define $\mu_i(g)$ as the logarithm of the $i$\textsuperscript{th} eigenvalue of $\sqrt{g g^T}$ (in decreasing order). We define the \emph{simple root critical exponent} of $\G$ by
\[\delta_{1,2}(\G) = \limsup_{R\to +\infty} \frac{1}{R} \log \left(\Card \{\gamma \in \G \mid \mu_1(\gamma) - \mu_2(\gamma) \leq R\}\right)\]
and the \emph{highest weight critical exponent} of $\G$ by
\[\delta_{1,n}(\G) = \limsup_{R\to +\infty} \frac{1}{R} \log \left(\Card \{\gamma \in \G \mid \mu_1(\gamma) - \mu_n(\gamma) \leq R\}\right)~.
\]

These critical exponents are relevant for different reasons: the projective Anosov property means that $\mu_1(\gamma) - \mu_2(\gamma)$ grows linearly with the word length of $\gamma$, so $\delta_{1,2}(\G)$ can be seen as a ``measure'' of the Anosov property. The critical exponent $\delta_{1,n}(\G)$ is the critical exponent associated to the Hilbert metric on $\SL(n,\R)/\SO(n)$ seen as the projectivization of the cone of positive definite quadratic forms on $\R^n$. Our main result compares these two critical exponents with the Hausdorff dimension of $\xi^{sym}(\partial_\infty \G)$ with respect to a Riemannian metric on $\proj (\R^n) \times \proj({\R^n}^*)$.

Our first comparison result between Hausdorff dimensions concerns strongly projectively convex-cocompact subgroups of $\SL(n,\R)$, introduced by Crampon and Marquis \cite{crampon2014finitude}. It is shown in \cite{DGK} that these groups are projective Anosov.

\begin{theo} \label{theo-MainTheorem}
Let $\Gamma$ be a strongly projectively convex-cocompact subgroup of $\SL(n,\R)$. Then
\[2 \delta_{1,n}(\Gamma) \leq \DimH(\xi^{sym}(\partial_\infty \G)) \leq \delta_{1,2}(\Gamma)~.\]
\end{theo}

For projective Anosov subgroups that are not convex-cocompact, composing with the representation of $\SL(n,\R)$ into $\SL(\mathrm{Sym}^2(\R^n))$ gives the following weaker result:
\begin{coro} \label{coro-NotConvexCocompact}
Let $\Gamma$ be a projective Anosov subgroup of $\SL(n,\R)$. Then
\[\delta_{1,n}(\Gamma) \leq \DimH(\xi^{sym}(\partial_\infty \Gamma)) \leq \delta_{1,2}(\Gamma)~.\]
\end{coro}

Note that Theorem \ref{theo-MainTheorem} is ``sharp'' in the sense that, when $\Gamma$ is a convex cocompact subgroup in $\SO(n-1,1)\subset \SL(n,\R)$, we have
\[2\delta_{1,n}(\Gamma) = \DimH(\xi^{sym}(\partial_\infty\Gamma)) = \delta_{1,2}(\Gamma)~.\]

Corollary \ref{coro-NotConvexCocompact} is weaker since $\delta_{1,n}(\Gamma)$ is always less than or equal to $\frac{1}{2} \delta_{1,2}(\G)$. However, it cannot be sharpened in full generality. For instance, let $\Gamma$ be a cocompact lattice in $\SL(2,\R)$ and let $\rho_{irr}$ and $\rho_{red}$ denote respectively the irreducible and reducible representations of $\SL(2,\R)$ into $\SL(3,\R)$. Then $\rho_{irr}(\Gamma)$ and $\rho_{red}(\Gamma)$ are projective Anosov with limit set a smooth curve (of Hausdorff dimension $1$). However, their critical exponents differ:
\begin{itemize}
\item $\rho_{irr}(\Gamma)$ is convex cocompact and
\[2\delta_{1,3}(\rho_{irr}(\G)) = \delta_{1,2}(\rho_{irr}(\G)) = 1~.\]
\item $\rho_{red}(\Gamma)$ is not convex cocompact and
\[\delta_{1,3}(\rho_{red}(\G)) = \frac{1}{2} \delta_{1,2}(\rho_{red}(\G)) = 1~.\]
\end{itemize}

Similar results where obtained simultaneously by Pozzetti--Sambarino--Wienhard \cite{PSW19} and, shortly after the first version of this paper appeared on arXiv, by Dey and Kapovich \cite{DeyKapovich}.

Let us now discuss further our main theorem.

\subsubsection*{Lower inequality}

The main motivation for the lower inequality in Theorem~\ref{theo-MainTheorem} was to generalize the following theorem of Crampon:

\begin{theo}[\cite{crampon2011dynamics}]
Let $\G \subset \SL(n,\R)$ be a Gromov hyperbolic group acting properly discontinuously and cocompactly on a strictly convex open domain $\Omega$ in $\proj(\R^n)$. Then
\[2\delta_{1,n} \leq n-2~,\]
with equality if and only if $\Gamma$ is conjugate to a subgroup of $\SO(n-1,1)$ (in which case $\Omega$ is projectively equivalent to the hyperbolic space $\Hyp^{n-1}$).
\end{theo}

In that case, $\Gamma$ is projective Anosov, $\xi(\partial_\infty \G)$ is the boundary of $\Omega$ and $\xi^*(\partial_\infty \G)$ the boundary of the dual convex set. We will show (Theorem \ref{theo - HD Symmetric boundary convex}) that $\xi^{\textit{sym}}(\partial_\infty(\G))$ is a Lipschitz manifold of dimension $n-2$, so that $\DimH(\partial_\infty \G) = n-2$. Therefore, Theorem \ref{theo-MainTheorem} recovers Crampon's inequality as a particular case.

For $\Gamma\subset \SL(n,\R)$ strongly projectively convex-cocompact, we initially hoped to prove the a priori stronger inequality
\[2 \delta_{1,n}(\Gamma) \leq \DimH(\xi(\partial_\infty \G)).\] But several attempts with slightly different methods always led to a ``symmetric'' version of the limit set. This raised the following question:
\begin{question} \label{ques-HaudorffDimensionSymmetric}
Let $\G \subset \SL(n,\R)$ be a projective Anosov subgroup. Do we have
\[\DimH(\xi(\partial_\infty \G)) = \DimH(\xi^*(\partial_\infty \G)) = \DimH(\xi^{sym}(\partial_\infty \G))~?\]
\end{question}
While our na\"ive intuition leaned towards a positive answer, the following case might actually provide a counter-example: Let $\Gamma$ be a lattice in $\SL(2,\R)$, $u:\G \to \R^2$ a function satisfying the cocycle relation
\[u(\gamma \eta) = u(\gamma) + \gamma \cdot u(\eta)~,\]
and let $\rho_u$ be the representation of $\Gamma$ into $\SL(3,\R)$ given by
\[\rho_u(\gamma) = \left( \begin{matrix} \gamma & u(\gamma) \\ 0 & 1\end{matrix} \right)~.\]

Then $\rho_u(\G)$ is projective Anosov\footnote{Indeed, for $u$ small enough, this follows from the openness of the Anosov property, and moreover, $\rho_u$ is conjugated to $\rho_{\epsilon u}$ for all $\epsilon>0$.}. Let $\xi_u: \partial_\infty \G \to \proj(\R^3)$ and $\xi_u^*: \partial_\infty \G \to \proj({\R^3}^*)$ denote the boundary maps associated to $\rho_u(\G)$. Then $\xi_u(\partial_\infty\G) = \xi_0(\partial_\infty \G)$ is a projective line. On the other side, the dual limit set $\xi_u^*(\partial_\infty \G)$ is not a projective line as soon as $u$ is not a coboundary\footnote{Recall that a cocycle $u$ is a coboundary if there exists $v\in \R^2$ such that $u(\gamma) = \gamma \cdot v - v$ for all $\gamma \in \G$.} and it could plausibly have Hausdorff dimension $>1$. 

There are situations where the equality in Question \ref{ques-HaudorffDimensionSymmetric} is known to be true: If $\G$ preserves a non-degenerate quadratic form $\mathbf q$ on $\R^n$, then $\xi^*$ is the image of $\xi$ by the isomorphism $\R^n\simeq {\R^n}^*$ defined by $\mathbf q$, and therefore
\[\DimH(\xi(\partial_\infty \G)) = \DimH(\xi^*(\partial_\infty \G)) = \DimH(\xi^{sym}(\partial_\infty \G))~.\]
In that case we also have that $\mu_n(\gamma) = - \mu_1(\gamma)$ for all $\gamma \in \G$, so that
\[2\delta_{1,n}(\G) = \delta_1(\G) \overset{\textrm{def}}{=} \limsup_{R\to +\infty} \frac{1}{R} \log \left(\Card \{\gamma \in \G \mid \mu_1(\gamma) \leq R\}\right)~.\]

Those strongly projectively convex-cocompact groups preserving a non degenerate quadratic form are precisely the $\Hyp^{p,q}$-convex cocompact groups introduced in \cite{DGKSOpq}, whose critical exponent was studied by the first two authors in \cite{GM18}. There, the authors introduce a \emph{pseudo-hyperbolic critical exponent} $\delta_{\Hyp^{p,q}}(\G)$ and prove that it coincides with $\delta_1(\G)$. Theorem \ref{theo-MainTheorem} thus gives an alternative proof of the inequality 
\[\delta_{\Hyp^{p,q}}(\G) \leq \DimH(\xi(\partial_\infty\Gamma))\]
in \cite[Theorem 1.2]{GM18}.

A rigidity statement in that context was obtained by Collier--Tholozan--Toulisse in \cite{CCT17} for $\Hyp^{2,q}$-convex cocompact surface groups, which are the images of fundamental groups of closed surfaces by maximal representations into $\SO(2,q+1)$. Their limit set is a Lipschitz curve (of Hausdorff dimension $1$), and they prove that the critical exponent $\delta_1$ is $\leq 1$, with equality if and only if the group is contained in $\SO(2,1)\times \SO(q)$ (up to conjugation and finite index). Together with Crampon's theorem, this leads us to formulate the following conjecture:

\begin{conj}
Let $\Gamma \subset \SL(n,\R)$ be a strongly projectively convex cocompact subgroup. If $2\delta_{1,n} = \DimH(\xi^{sym}(\partial_\infty \G))$, then $\G$ is conjugated to a subgroup of $\SO(n-1,1)$.
\end{conj}

Note that Potrie--Sambarino proved in \cite{potrie2014} a similar but stronger inequality for Hitchin representations of surface groups. If $\G$ is the fundamental group of a closed surface and $\rho:\G\to \SL(n,\R)$ is a Hitchin representation, then $\rho(\G)$ is projective Anosov and $\xi_\rho(\partial_\infty \G)$ is a $\mathcal{C}^1$ curve, of Hausdorff dimension $1$. However, they prove that
\[2\delta_{1,n}(\rho(\G)) \leq \frac{2}{n-1}~.\]
with equality if and only if $\rho = m_{irr} \circ j$ where $j:\G \to \SL(2,\R)$ is Fuchsian and $m_{irr}: \SL(2,\R) \to \SL(n,\R)$ is irreducible.

\subsubsection*{Upper inequality}

The upper inequality $\DimH(\xi^{sym}(\partial_\infty \G)) \leq \delta_{1,2}(\G)$ is proven independently by Pozzetti--Sambarino--Wienhard in \cite{PSW19}. There, they also find a sufficient criterion for this inequality to be an equality. This criterion is satisfied by many families of Anosov groups, showing in particular that the equality can be stable under small deformations of $\G$. Their work generalizes a result of Potrie--Sambarino for surface groups embedded in $\SL(n,\R)$ via a Hitchin representation. They are in stark contrast with the rigidity phenomena for the $\delta_{1,n}$ discussed above.

Here we merely give an example where equality holds:
\begin{theo} \label{theo - simple root exponent tensor product Fuchsian}
Let $\G$ be the fundamental group of a closed surface of genus greater than $1$ and let $j_1$ and $j_2$ be two Fuchsian representations of $\G$ into $\SL(2,\R)$. Then $j_1\otimes j_2(\G) \subset \SL(2,\R)\times \SL(2,\R) \subset \SL(4,\R)$ is projective Anosov, $\xi_{j_1\otimes j_2}^{sym}(\partial_\infty \G)$ is a Lipschitz curve and
\[\delta_{1,2}(j_1\otimes j_2(\G)) = 1~.\]
\end{theo}

The groups to which this theorem applies are the fundamental groups of globally hyperbolic Cauchy compact anti-de Sitter spacetimes studied by Mess \cite{mess2007lorentz}. They form a connected component in the space of surface groups embedded in $\SL(2,\R)\times \SL(2,\R) \simeq \SO(2,2)$. This class of examples is not covered by the main result of Pozzetti--Sambarino--Wienhard \cite{PSW19}\footnote{It is however a particular case of the main result of their new work \cite{PSW19b}.}.

On the other hand, a Fuchsian group of $\SL(2,\R)$ embedded reducibly in $\SL(3,\R)$ gives an example where $\DimH(\xi^{sym}(\G)) < \delta_{1,2}$. Determining a necessary and sufficient criterion for the equality to hold seems difficult.

\subsection{Further results and strategy of the proof}

Section \ref{sec - Background} introduces the background of this work (Anosov groups, critical exponents, convex projective geometry). In particular, we explain in Section \ref{subsec - fundamental weights and fundamental representations} how any Anosov subgroup of a semisimple Lie group $G$ defines a projective Anosov group after composing with a suitably chosen linear representation $\tau$ of $G$. In Section \ref{subsec - Critical exponents symmetric space} we introduce the various critical exponents $\delta_\alpha(\G)$ of an Anosov group $\G\subset G$ (depending on a choice of linear form $\alpha$ on a Cartan subalgebra of $G$) and we prove a fairly general result about when these critical exponents coincide with the corresponding \emph{entropy} $h_\alpha(\Gamma)$ (see Definition \ref{def - entropy}).

In Section \ref{sec - lower bound}, we prove the lower inequality of Theorem \ref{theo-MainTheorem}. The proof follows a strategy similar to that of \cite{GM18}, with the Hilbert geometry of $\Omega$ replacing the pseudo-Riemannian geometry of $\Hyp^{p,q}$. More precisely, we first notice that $2\delta_{1,n}(\Gamma)$ coincides with the critical exponent of $\Gamma$ acting on $\Omega$ equiped with its \emph{Hilbert metric}. Applying a theorem of Coornaert (Theorem \ref{theo - coornaert}), we deduce the equality
\[2\delta_{1,n}(\Gamma) = \DimH(\partial_\infty \Gamma, d_{\textit{Gromov}})~,\]
the Hausdorff dimension of $\partial_\infty \Gamma$ equipped with \emph{Gromov's quasidistance} $d_{\textit{Gromov}}$ (see Definition \ref{def - gromov distance}). Finally, we prove that, for any Gromov hyperbolic convex set $\Omega \subset \proj(\R^n)$, there exists a constant $C$ such that, for any $p,q\in \partial \Omega$, 
\[d_{\textit{Gromov}}(p,q) \leq C \sqrt{d(p,q) d^*(p^*,q^*)}~,\]
where $p^*$ and $q^*$ denote respectively the hyperplanes tangent to $\Omega$ at $p$ and $q$ and $d$ and $d^*$ denote Riemannian distances on $\proj(\R^n)$ and $\proj({\R^n}^*)$ respectively. The inequality
\[\DimH(\partial_\infty \Gamma, d_{\textit{Gromov}}) \leq \DimH(\xi_{\textit{sym}}(\partial_\infty \Gamma))\]
then follows from elementary comparison results between Hausdorff dimensions.

Section \ref{sec - upper bound} is devoted to the proof of the second upper inequality in Theorem \ref{theo-MainTheorem}. Let $\Gamma$ be a projective Anosov subgroup of $\SL(n,\R)$. Roughly speaking, $e^{\mu_2(\gamma)-\mu_1(\gamma)}$ controls the Lipschitz factor of an element $\gamma \in \Gamma$ on a large part of $\proj(\R^n)$. This allows us to cover $\xi(\partial_\infty \Gamma)$ by balls with radii controlled by $e^{\mu_2(\gamma)-\mu_1(\gamma)}$, giving an upper bound on its Hausdorff dimension.

\subsection*{Acknowledgements}
While writing this paper, we have been informed that Beatrice Pozzetti, Andres Sambarino and Anna Wienhard were working on similar results. We thank them for sharing their work in progress.

Nicolas Tholozan's research is partially supported by the Agence Nationale de la Recherche through the grant DynGeo
(ANR-11-BS01-013)

Olivier Glorieux acknowledges support from the European Research Council (ERC) under the European Union’s Horizon 2020 research and innovation programme (ERC starting grant DiGGeS, grant agreement No 715982).

\section{Background} \label{sec - Background}

In this section, we introduce some background to our work. We start by stating a few basic properties of Hausdorff dimension (Section \ref{subsec - hausdorff dimension}). In Section \ref{subsec - lie theory} we recall some fundamentals of Lie theory. In Section \ref{subsubsec-Anosov representation}, we introduce Anosov subgroups of semisimple Lie groups and explain how to embed then as projective Anosov subgroups of $\SL(n,\R)$. In Section \ref{subsec - Critical exponents symmetric space} we introduce their various critical exponents and entropies, and prove some comparison results between them. Finally, in Section \ref{subsec - projectively cc representations}, we introduce \emph{strongly projectively convex-cocompact groups}, following \cite{DGK}.

\subsection{Hausdorff dimension} \label{subsec - hausdorff dimension}

Let $(X,d)$ be a metric space. For $s>0$, the  Hausdorff measure of $X$ of dimension $s$ is defined by 
$$H^s(X,d) = \lim_{\epsilon\tv 0} H^s_\epsilon(X,d)~,$$
where
$$H^s_\epsilon(X,d) = \inf  \{ \sum_i r_i^s  \, | \, X \subset \bigcup_{i\in I} B(x_i, r_i), r_i \leq \epsilon\}~.$$
Here, the infimum is taken over all countable covers of $X$ by balls of radius less than $\epsilon$. 
One can show that there exists a critical parameter $s_0>0$ such that $H^s(X,d) = +\infty$ for all $s<s_0$ and $H^s(X,d) = 0$ for all $s>s_0$. This number $s_0$ is called the Hausdorff dimension of $(X,d)$ and is denoted by $\DimH(X,d)$.

In this paper, we will compare the Hausdorff dimension of different metrics on a compact set. The classical following proposition summarizes the comparison properties that we will need. 
\begin{prop} \label{p:GeneralComparisonHDim}
Let $d$ and $d'$ be two distances on a space $X$. If there exists $C$ and $\alpha>0$ such that
\[d' \leq C d^{\alpha}~,\]
then
\[\DimH(X,d') \leq \frac{1}{\alpha} \DimH(X,d)~.\]
\end{prop}

In particular, if $d$ and $d'$ are bi-Lipschitz, then \[\DimH(X,d) = \DimH(X,d')~.\]
Assume now that $X$ is a compact subset of a smooth manifold $M$. Any two Riemannian metrics on $M$ are bi-Lipschitz equivalent in a neighbourhood of $X$. Hence the Hausdorff dimension of $X$ with the induced distance is independent of the choice of such a metric. We denote this Hausdorff dimension by $\DimH(X)$ and we have: 
\[\DimH(X)= \DimH(X,d)\]
where $d$ is the distance induced by any Riemannian metric on $M$.

\subsection{Cartan and Jordan projections} \label{subsec - lie theory}

\subsubsection{Cartan subspaces and restricted roots}
We present in this subsection the basic structure theory of semi-simple real Lie groups. A detailed exposition of this theory can be found in \cite{eberlein1996geometry}.

Let $G$ be a real semisimple Lie group with finite center, $K$ a maximal compact subgroup of $G$ and $X=G/K$ the symmetric space of $G$. We denote by $\fg$ the Lie algebra of $G$ and by $\fk\subset \fg$ the Lie algebra of $K$. Let $\fp$ denote the orthogonal of $\fk$ with respect to the Killing form of $\fg$. A \emph{Cartan subspace} $\fa$ is a maximal Abelian subalgebra of $\fp$.

A \emph{restricted root} is a non-zero linear form $\alpha$ on $\fa$ for which there exists $u\in \fg \backslash \{0\}$ such that
\[\ad_a(u) = \alpha(a) u\]
for all $a\in \fa$. We will denote by $\Delta$ the set of restricted roots. From now on,  restricted roots will be called roots for simplicity.

The \emph{Weyl group} $W(\fa)$ is the finite group $N(\fa)/Z(\fa)$, where $N(\fa)$ and $Z(\fa)$ denote respectively the normalizer and the centralizer of $\fa$ in $K$. The kernels of the restricted roots cut $\fa$ into fundamental domains for the action of $W(\fa)$. Choosing a connected component of $\fa \backslash \bigcup_{\alpha \in \Delta} \ker \alpha$, we define the set of \emph{positive roots} $\Delta^+$ as the roots that are positive on this connected component, and the \emph{Weyl chamber} as the closure of this connected component, i.e.
\[\fa_+ = \{b\in \fa \mid \alpha(b) \geq 0 \textrm{ for all $\alpha\in \Delta^+$}\}~.\]
The Weyl chamber is a convex cone. We denote by $\fa_+^*$ the dual convex cone, i.e.
\[\fa_+^* = \{\alpha \in \fa^* \mid \alpha_{\vert \fa_+} \geq 0\}~.\]

With those choices, the \emph{simple roots} are the positive roots that are not a positive linear combination of other positive roots. They form a basis of $\fa^*$. We denote by $\Delta_s$ the set of simple roots. 

Finally there is a unique element $w\in W(\fa)$ such that $-w$ preserves $\fa^+$. The transformation $-w$ is an involution called the \emph{opposition involution} and denoted by $\op$. The opposition involution preserves $\Delta_s$.

\paragraph{\textit{Main example}} The main example we will be interested in here is when $G$ is the group $\SL(n,\R)$. A canonical choice for a maximal compact subgroup $K$ is the subgroup $\SO(n,\R)$ of orthogonal matrices. The symmetric space $X_n = \SL(n,\R)/\SO(n)$  can be identified with the space of scalar products on $\R^n$ up to scaling, with the standard scalar product as base point $o$. 

The Lie algebra $\fk$ is the space of anti-symmetric matrices and its orthogonal $\fp \subset \mathfrak{sl}(n,\R)$ is the space of symmetric matrices of trace $0$. A canonical choice of Cartan subspace is $\fa= \{\textrm{Diagonal matrices of trace $0$}\}= \{ \diag(\lambda_1, ... , \lambda_n), \sum_i \lambda_i=0\} $. The Weyl group is the symmetric group $\mathfrak{S}_n$ acting by permuting the eigenvalues. Denote by $\epsilon_i\in \fa^*$ the linear form on $\fa$ corresponding to the $i$-th eigenvalue. The restricted roots are the $\alpha_{i,j}=\epsilon_i - \epsilon_j$, for $1\leq i, j \leq n$.  A canonical choice of Weyl chamber is $\fa^+ =\{\textrm{Diagonal matrices with ordered eigenvalues}\} =\{ \diag(\lambda_1, ... , \lambda_n), \, \lambda_1 \geq \lambda_2\geq \ldots \geq \lambda_n, \,  \sum_i \lambda_i=0\}$, with associated set of positive restricted roots $\{\alpha_{i,j}, 1\leq i < j \leq n\}$. The simple roots are the roots $\alpha_{i,i+1}$, $1\leq i \leq n-1$. Finally, the opposition involution $\op$ maps $\diag(\lambda_1,\ldots, \lambda_n)$ to $\diag(-\lambda_n,\ldots, -\lambda_1)$.

\subsubsection{Cartan projections}

From now on, we always assume a fixed choice of 
\begin{itemize}
\item a maximal compact subgroup $K$,
\item a Cartan subalgebra $\fa \subset \fp$,
\item a Weyl chamber $\fa_+ \subset \fa$, with associated positive roots $\Delta^+$ and simple roots $\Delta_s$.
\end{itemize}
Note that the choice of a maximal compact subgroup $K$ corresponds to the choice of a base point $o = \mathrm{Fix}(K)$ in the symmetric space $X$.

\begin{theo}[Cartan decomposition]
For every $g\in G$, there is a unique vector $\mu(g) \in \fa_+$ such that
\[g = k \exp(\mu(g)) k'\]
for some $k,k'\in K$. The map $\mu: G \to \fa_+$ is called the \emph{Cartan projection}.
\end{theo}

\begin{rmk}
The Cartan projections of $g$ and its inverse are related by the following formula:
\[\mu(g^{-1}) = \op(\mu(g)).\]
This relation characterizes the opposition involution.
\end{rmk}

The Cartan projection allows to define a ``vector valued distance'' on the symmetric space $X$. If $x$ and $y$ are two points in $X$, we define
\[\mu(x,y) = \mu(g^{-1}h),\]
where $g$ and $h$ are elements of $G$ such that $g\cdot o = x$ and $h\cdot o = y$. This is a vector valued distance in the following sense: if $\norm{\cdot}$ is a $W(\fa)$-invariant norm on $\fa$, then
\[(x,y) \mapsto \norm{\mu(x,y)}\]
is a $G$-invariant Finsler distance on $X$. In particular, if $\norm{\cdot}_{\textit{eucl}}$ is the Euclidean norm on $\fa$ given by the Killing form, then
\[\norm{\mu(x,y)}_{\textit{eucl}} = d_R(x,y),\]
where $d_R$ is the symmetric Riemannian distance of $X$.

Benoist showed  that the Cartan projection satisfies a generalized triangle inequality: 
\begin{prop}\cite{benoist1997proprietes}\label{prop - Benoist cartan projection continuous}
For every compact subset $L$ of $G$, there is a constant $C>0$ (depending on $L$) such that
\[\norm{\mu(lgl') - \mu(g)}_{\textit{eucl}} \leq C\]
for all $g\in G$ and all $l,l'\in L$.
\end{prop}
In particular, given two points $x$ and $y\in X$, there is a constant $C$ (depending on $x$ and $y$) such that
\[\norm{\mu(x,z) - \mu(y,z)}_{\textit{eucl}} \leq C\]
for all $z\in X$.

\subsubsection{Jordan projections}
For a restricted root $\alpha\in \Delta$  we denote by $\mathfrak{g}_\alpha := \{ u \in \mathfrak{g}, \ad_a(u) = \alpha(a) u , \, \forall a \in  \fa \} $ the corresponding eigenspace. We denote by $A^+ := \exp(\fa^+)$ and $N=\exp(\oplus_{\alpha \in \Delta^+} \mathfrak{g}_\alpha) $. 

An element of $G$ is called \emph{elliptic} (resp. \emph{hyperbolic}, \emph{unipotent}) if it is conjugated to an element of $K$ (resp. $A^+$, $N$).

\begin{theorem}[Jordan decomposition]\cite[ Theorem 2.19.24]{helgason1978differential}
For all $g\in G$, there is a unique triple $(g_e, g_h, g_p)$ of commuting elements, such that  $g_e$ is elliptic, $g_h$ hyperbolic and $g_p$ unipotent, that satisfies: $g=g_eg_hg_p$. 
\end{theorem}

\begin{defi}
The Jordan projection of $g$ is the element $\lambda(g)\in \fa^+$ such that $g_h$ is conjugated to $\exp(\lambda(g))$. 
\end{defi}
While the Cartan projection depends on the choice of a base point in~$X$, the Jordan projection is a conjugacy invariant. One has the following alternative definition of $\lambda(g)$:

\begin{prop}
For every $g\in G$,
\[\lambda(g) = \lim_{n\to +\infty} \frac{1}{n}\mu(g^n)\]
\end{prop}

\begin{rmk}
Similarly to the Cartan projection, we have the following relation:
\[\lambda(g^{-1}) = \op(\lambda(g))~.\]
\end{rmk}

\paragraph{\textit{Main example}}

The Cartan decomposition for $\SL(n,\R)$ is usually called the polar decomposition, and the Cartan projection associates to a matrix $g\in \SL(n,\R)$ the logarithm of the eigenvalues of $\sqrt{g^t g}$ in decreasing order. We will denote by $\mu_i(g) = \epsilon_i(\mu(g))$ the $i$-eigenvalue of the Cartan projection of~$g$.  

The decomposition $g = g_e g_h g_p$ in that case is sometimes called the Dunford decomposition. The Jordan projection associates to $g$ the logarithms of the moduli of the complex eigenvalues of $g$, in decreasing order. We will denote similarly by $\lambda_i(g)= \epsilon_i(\lambda(g))$ the $i$\textsuperscript{th} eigenvalue of the Jordan projection of $g$.

\subsection{Anosov groups}\label{subsubsec-Anosov representation}

Anosov subgroups of higher rank Lie groups have been introduced by Labourie \cite{labourie2006anosov} as a reasonable generalization of convex-cocompact subgroups in rank 1. The original definition for deformations of uniform lattices in rank 1 was extended by Guichard and Wienhard to Gromov hyperbolic groups. More recently, Gueritaud--Guichard--Kassel--Wienhard \cite{gueritaud20017anosov} and Kapovich--Leeb--Porti \cite{kapovich2017anosov} independently gave a characterization of Anosov subgroups in terms of their Cartan projections. While the first team assumes \emph{a priori} that the group is hyperbolic, the second team shows moreover that their condition implies Gromov hyperbolicity. Here, following \cite{Guichard}, we use their characterization as a definition.

Let $G$ be a semisimple Lie group. Fix a choice of $K$, $\fa$ and $\fa_+$ as before.  Let $\Theta$ be a non-empty subset of the set of simple roots $\Delta_s$.

\begin{defi}
A finitely generated group $\G \subset G$ is called \emph{$\Theta$-Anosov} if there exist constants $\alpha, A>0$ such that
\[\theta(\mu(\g)) \geq \alpha \vert \g \vert - A \]
for all $\g\in \G$ and all $\theta \in \Theta$.
(Here, $\vert g \vert$ denotes the word length of $g$ with respect to a finite generating set.)
\end{defi}
The definition implies in particular that $\G$ is discrete and quasi-isometrically embedded in $G$. One of the nice features of this definition is that it forces $\Gamma$ to have some ``negatively curved behaviour'':

\begin{theo}\cite[Theorem 6.15]{kapovich2018anosov}
Let $\G \subset G$ be a $\Theta$-Anosov subgroup, for some non-empty subset $\Theta$ of $\Delta_s$. Then $\G$ is Gromov hyperbolic.
\end{theo}

\begin{rmk}
Since $\Gamma$ is invariant by $g\mapsto g^{-1}$ this definition readily implies that a $\Theta$-Anosov subgroup is also $\op(\Theta)$-Anosov, and thus $\Theta^{sym}$-Anosov, where $\Theta^{sym} = \Theta \cup \op(\Theta)$. There is thus no loss of generality in assuming that $\Theta$ is invariant by the opposition involution.
\end{rmk}

\paragraph{{\it Main example}}

Let us describe more properties of Anosov subgroups in a specific case. In the next section, we will explain how to reduce the general case to this specific case.
\begin{defi}
A finitely generated group $\G \subset \SL(n,\R)$ is called \emph{projective Anosov} if there exist constants $\alpha, A>0$ such that
\[\mu_1(\g) - \mu_2(\g) \geq \alpha \vert \g \vert - A,\]
for all $\g\in \G$.
\end{defi}

\begin{rmk}
By definition, a projective Anosov subgroup is  $\Theta$-Anosov for $\Theta =\{\alpha_{1,2}\}$. Since the opposition involution sends $\alpha_{1,2}$ to $\alpha_{n-1, n}$, projective Anosov subgroups are actually $\Theta^{sym}$-Anosov for $\Theta^{sym} =\{\alpha_{12}, \alpha_{n-1,n}\}$ 
\end{rmk}

The group $\SL(n,\R)$ acts on the projective space $\proj(\R^n) = \{\textrm{lines in }\R^n\}$ and on the ``dual'' projective space $\proj({\R^n}^*)= \{\textrm{hyperplanes in }\R^n\}$. Recall that the \emph{Gromov boundary} $\partial_\infty \G$ of a Gromov hyperbolic group $\G$ is a compact metrizable space on which $\G$ acts by homeomorphisms. The following theorem says that the Gromov boundary of a projective Anosov subgroup is ``realized'' in the projective space:

\begin{theo} \label{theo - AnosovLimitMaps}[\cite{labourie2006anosov},\cite{kapovich2017anosov}] \label{t:BoundaryMapsAnosov}
Let $\G\subset \SL(n,\R)$ be a projective Anosov subgroup. Then there exist $\G$-equivariant continuous maps $\xi: \partial_\infty \Gamma \to \proj(\R^n)$ and $\xi^*: \partial_\infty \Gamma \to \proj({\R^n}^*)$ such that
\[\xi(x)\subset \xi^*(y) \Longleftrightarrow x = y~.\]
These maps are moreover \emph{strongly dynamics preserving} in the following sense: for every sequence $(\g_n)\in \Gamma^\N$ such that $\g_n\underset{n\to +\infty}{\longrightarrow} x \in \partial_\infty \Gamma$ and $\g_n^{-1}\underset{n\to +\infty}{\longrightarrow} y \in \partial_\infty \Gamma$, and for every $v\in \proj(\R^n) \backslash \xi^*(y)$,
\[\g_n v \underset{n\to +\infty}{\longrightarrow} x~.\]
\end{theo}
A consequence of the strongly dynamics preserving property is that every element $\g\in \G$ of infinite order is \emph{proximal}: its action on $\proj(\R^n)$ has a unique attracting fixed point $\xi(\g_+)\in \proj(\R^n)$, with basin of attraction $\proj(\R^n) \backslash \xi^*(\g_-)$ (here $\g_+$ and $\g_-$ denote respectively the attracting and repelling fixed points of $\g$ in $\partial_\infty \Gamma$).

\subsubsection{Fundamental weights and fundamental representations} \label{subsec - fundamental weights and fundamental representations}
Here, we explain how to interpret the $\Theta$-Anosov property as several projective Anosov properties, via linear representations of the Lie group $G$. The content of this section is already described in \cite[Section 3]{gueritaud20017anosov}.
 
Let $\scal{\cdot}{\cdot}$ denote a scalar product on $\fa^*$ invariant under the Weyl group action. 
\begin{defi}
The \emph{fundamental weight} $w_\theta$ associated to a simple root $\theta$ is the unique element of $\fa^*$ such that
\[2\frac{\scal{w_\theta}{\alpha}}{\scal{\alpha}{\alpha}}= \delta_{\alpha,\theta}~,\]
where $\delta_{\alpha,\theta}$ is the Kronecker symbol. 
\end{defi}

The classical representation theory of semi-simple Lie algebras gives the following:
\begin{lemme}
For every $\theta \in \Delta_s$ there is an integer $n_\theta \geq 2$, an integer $k_\theta$ and an irreducible representation $\rho_\theta: G\to \SL(n_\theta,\R)$ mapping $K$ into $\SO(n_\theta)$ and such that
\begin{itemize}
    \item $\theta(\mu(g))= \mu_1(\rho_\theta(g)) - \mu_2(\rho_\theta(g))$,
    \item $k_\theta w_\theta(\mu(g)) = \mu_1(\rho_\theta(g))$,
    \item $\theta\circ \op(\mu(g))= \mu_{n_\theta-1}(\rho_\theta(g)) - \mu_{n_\theta}(\rho_\theta(g))$,
    \item $w_{\theta\circ \op}(\mu(g)) = -\mu_{n_\theta}(\rho_\theta(g))$
\end{itemize}
for all $g\in G$. We call this $\rho_\theta$ the \emph{fundamental representation}\footnote{These properties actually do not characterize a representation, but they do if we assume moreover that $k_\theta$ is minimal. When $G$ is the split real form of a complex Lie group (such as $\SL(n,\R)$), we can have $k_\theta= 1$.}.
\end{lemme}

\begin{rmk}
The equalities above hold when replacing Cartan projections with Jordan projections.
\end{rmk}

\begin{rmk}
The fundamental representation $\rho_{\theta\circ \iota}$ is dual to the the representation $\rho_\theta$.
\end{rmk}

The following proposition easily follows from the definitions of Anosov representation:

\begin{prop}
Let $\Gamma$ be a finitely generated subgroup of $G$ and $\Theta$ be a non-empty subset of $\Delta_s$. Then $\Gamma$ is $\Theta$-Anosov if and only if $\rho_\theta(\Gamma)$ is projective Anosov for all $\theta \in \Theta$.
\end{prop}

\begin{example}
For $G= \SL(n,\R)$, let $\theta_i$ denote the simple root $\alpha_{i,i+1}$ then the fundamental weight $w_i = w_{\theta_i}$ associated to $\theta_i$ is the linear form $\epsilon_1+\ldots + \epsilon_i$, and the fundamental representation $\rho_{\theta_i}$ is the representation of dimension $\left( \begin{matrix} n \\ i \end{matrix} \right)$ given by the action of $\SL(n,\R)$ on $\Lambda^i(\R^n)$. 
\end{example}

Taking tensor products of fundamental representations, one obtains representations for which $\mu_1- \mu_2$ captures the behaviour of several simple roots at once. Given $\Theta$ a non-empty subset of simple roots, denote by $\rho_\theta: G \to \SL(V_\theta)$ the fundamental representations associated to each $\theta\in \Theta$ and by
\[\rho_\Theta = \bigotimes_{\theta \in \Theta} \rho_\theta : G \to \SL\left( \bigotimes_{\theta \in \Theta} V_\theta\right)\]
the tensor product representation.

\begin{prop} \label{prop - SimpleRootsTensorProduct}
For all $g\in G$, we have
\begin{itemize}
    \item $\mu_1(\rho_\Theta(g)) = \sum_{\theta \in \Theta} k_\theta w_\theta(\mu(g))$,
    \item $\mu_1(\rho_\Theta(g)) - \mu_2(\rho_\Theta(g)) = \inf_{\theta\in \Theta} \theta(\mu(g))$.
\end{itemize}
\end{prop}

As a corollary we obtain the following:
\begin{coro} \label{coro - anosov is projective anosov}
A subgroup $\G\subset G$ is $\Theta$-Anosov if and only if $\rho_\Theta(\G)$ is projective Anosov.
\end{coro}

The behaviour of the limit maps under these tensor products is given by the following proposition. Let $\G \subset G$ be a $\Theta$-Anosov subgroup and let $\xi_\theta: \partial_\infty \Gamma \to \proj(V_\theta)$ be the boundary map associated to $\rho_\theta(\G)$, seen as a projective Anosov subgroup of of $\SL(V_\theta)$.

\begin{prop} \label{prop - boundary map tensor product}
The boundary map $\xi_\Theta$ associated to $\rho_\Theta(\G)$ sends a point $x\in \partial_\infty \G$ to
\[\xi_\Theta(x) = \bigotimes_{\theta\in \Theta} \xi_\theta(x) \in \proj\left(\bigotimes_{\theta\in \Theta} V_\theta\right)~.\]
\end{prop}

\begin{rmk}
The boundary map $\xi_\Theta$ takes values in the algebraic set of pure tensors in $\proj\left(\bigotimes_{\theta\in \Theta} V_\theta\right)$ which is canonically isomorphic to $\prod_{\theta \in \Theta} \proj (V_\theta)$.
\end{rmk}

\begin{example} \label{example - Symmetric boundary map}
Let $\Gamma \subset \SL(n,\R)$ be a projective Anosov subgroup. Then $\Gamma$ is $\Theta$-Anosov with $\Theta = \{\alpha_{1,2}, \alpha_{n-1,n}\}$. Let $V$ denote the space $\R^n$ seen as the standard representation of $\SL(n,\R)$. Then $\rho_\Theta: \SL(n,\R) \to \SL(V\otimes V^*)$ is the tensor product of the standard representation and its dual.

If $\xi:\partial_\infty \G \to \proj (V)$ and $\xi^*:\partial_\infty \G \to \proj(V^*)$ denote the boundary maps from Theorem \ref{t:BoundaryMapsAnosov}, then the boundary map associated to $\rho_\Theta$ is the map
\[\xi^{sym} = (\xi, \xi^*): \partial_\infty \Gamma \to \proj(V)\times \proj(V^*) \subset \proj(V\otimes V^*)~.\]
\end{example}


For future use, we introduce the following notations. Given a subset $\Theta$ of $\Delta_s$, we define
\[C(\Theta) = \bigcup_{\theta \in \Theta}\{v\in \fa_+ \mid \theta(v) = 0\}\]
and 
\[C^*(\Theta) = \Span_{\R_+}(\Theta)~.\]
Define also
\[\fa_+(\Theta) = \fa_+ \setminus C(\Theta)\]
and 
\[\fa_+^*(\Theta) = \fa_+^* \setminus C^*(\Delta_s -\Theta) = \{\phi \in \fa_+^*\mid \phi_{\vert \fa_+(\Theta)} >0\}.\]

\begin{rmk}
The motivation to consider such a subset of linear forms comes from the counting of elements of the group, as we will see in the next section. For a $\Theta$-Anosov subgroup $\G$, we know that the Cartan projections of elements of $\G$ lie in a closed cone contained in $\fa_+(\Theta)$. In particular, $\phi(\mu(g))$ grows linearly with $|g|$ for all $\phi\in\fa_+^*(\Theta)$.
\end{rmk}

\paragraph{{\it Main example.}} Let $G$ be $\SL(n,\R)$ and $\Theta = \{\alpha_{1,2}, \alpha_{n-1,n}\}$. We then have
$$C(\Theta) = \{v\,  \in \fa_+\, |\, \alpha_{1,2} (v) = 0 \}\cup \{v\,  \in \fa_+\, |\,  \alpha_{n-1,n} (v) = 0 \}.$$

Thus the set $\mathfrak{a}_+(\Theta)$ consists of diagonal matrices for which there is a spectral gap between the two highest, and between the  two lowest eigenvalues. 
Finally, $\fa_+^*(\Theta)$ is the set of linear forms on $\fa$ which are strictly positive on $\fa^+$ except maybe on the walls of the Weyl chambers defined by the equality of the two highest (resp. smallest) eigenvalues. In coordinates this means that any linear form $\phi\in \fa_+^*(\Theta)$ can be written as $\phi= \sum_{i=1}^{n-1} x_i \alpha_{i,i+1}$, for $x_i \in \R^{n-1}$  where $x_i> 0$ for all $i \in \{ 2, ..., n-2\}$ and $x_i\geq 0$ for $i\in \{ 1, n-1\}$. 

\subsection{Critical exponents and entropies} \label{subsec - Critical exponents symmetric space}

The critical exponent of a discrete group of isometries of a metric space is the exponential growth rate of the orbit of a basepoint. In the case of a discrete subgroup $\G$ of a higher rank semisimple Lie group $G$ acting on its symmetric space $X$, one can define a critical exponent for each $G$-invariant distance on $X$, and more generally for every choice of a way of measuring the ``size'' of Cartan projections. Following Quint \cite{quint2002divergence}, we focus here on non-negative linear forms on the Weyl chamber.

\begin{defi}
Let $\G$ be a discrete subgroup of $G$ and $\varphi$ a linear form on $\fa$ which is non-negative on the Weyl chamber. We define the \emph{$\varphi$-critical exponent} of $\G$ as
\begin{eqnarray*}
\delta_\varphi(\G) &=& \limsup_{R\tv \infty} \frac{1}{R} \log \left(\Card \{ \g \in \G \, |\, \varphi(\mu(\g)) \leq R\} \right)\\
&=& \inf\{s>0\mid \sum_{\g\in \G} e^{-s\varphi(\mu(\g))} < +\infty\}~.
\end{eqnarray*}
\end{defi}

In full generality, $\delta_\varphi(\G)$ has no reason to be finite. However, for finitely generated groups, Quint showed in \cite{quint2002divergence} that $\delta_\phi(\G)$ is finite as soon as $\varphi$ is positive on the \emph{limit cone} of $\G$, defined as
\[\Cone(\G) = \bigcap_{n\in \N}\overline{\bigcup_{\g \in \G, \vert \g \vert \geq n} \R \mu(\g)}~.\]
Applying his results to the case of Anosov representations gives the following
\begin{prop}[Quint, \cite{quint2002divergence}]
Let $\Theta$ be a non-empty subset of $\Delta_s$. Then
\[\delta_\varphi(\G) < +\infty\]
for every linear form $\varphi$ in $\fa_+^*(\Theta)$ and every  $\Theta$-Anosov subgroup $\G$. Moreover, the map
\[\varphi \mapsto \delta_\varphi\]
is convex and homogeneous of degree $-1$ on $\fa_+^*(\Theta)$.
\end{prop}

In a similar way, one can consider the exponential growth rate of the Jordan projections.

\begin{defi} \label{def - entropy}
Let $\G$ be a discrete subgroup of $G$ and $\varphi$ a linear form on $\fa$ which is non-negative on the Weyl chamber. We define the \emph{$\varphi$-entropy} of $\G$ as
\begin{eqnarray*}
h_\varphi(\G) &=& \limsup_{R\tv \infty} \frac{1}{R} \log \Card \{ [\g] \in [\G] \mid  \varphi(\lambda(\g)) \leq R\} \\
&=& \inf\{s>0\mid \sum_{[\g]\in [\G]} e^{-s\varphi(\lambda(\g))} < +\infty\}~,
\end{eqnarray*}
where $[\G]$ denotes the set of conjugacy classes in $\G$.
\end{defi}

The term ``entropy'' comes from the analogy with the geodesic flow of a closed negatively curved manifold, whose closed orbits are in bijection with conjugacy classes in the fundamental group, and whose topological entropy equals the exponential growth rate of lengths of closed orbits. In the case where $\phi$ is a linear combination of the fundamental weights $w_\theta, \theta \in \Theta$, this is more than an analogy: one can associate to a $\Theta$-Anosov subgroup $\Gamma$ of $G$ a flow on a compact metric space, whose orbits are in bijection with conjugacy classes in $\Gamma$, and such that the length of the orbit associated to $g$ is given by $\phi(\lambda(g))$. This flow has a hyperbolicity property, and its topological entropy is $h_\phi$ (see for instance \cite{sambarino2014quantitative}).

For sufficiently nice discrete groups of isometries of a negatively curved manifold, the critical exponent equals the entropy. For a Zarisky dense $\Theta$-Anosov group, Sambarino obtained in  \cite{sambarino2014quantitative} precise counting estimates for 
\[\Card \{\gamma \in \Gamma \mid w_\theta(\mu(\gamma))\leq R\}~,\]
implying in particular that $h_\phi(\Gamma) = \delta_\phi(\Gamma)$ when $\phi$ is  a linear combination of the fundamental weights $\{w_\theta,\theta \in \Theta\}$. The tools he uses, however, do not seem to apply to simple root critical exponents in general. Here we prove the equality $\delta_\phi = h_\phi$ whenever we manage to generalize the classical arguments that work in negative curvature.

For the sake of clarity, let us first state our result in the main case of interest for us.

\begin{theorem}\label{theo-specialcaseSL(n,R)}
Let $\Gamma$ be a projective Anosov subgroup of $\SL(n,\R)$. Then 
\begin{itemize}
\item  $h_{1,2}(\Gamma)\leq \delta_{1,2}(\Gamma),$
\item $h_{1,n}(\Gamma)= \delta_{1,n}(\Gamma).$
\end{itemize}
If $\G$ is moreover Zariski dense in $\SL(n,\R)$, then 
\begin{itemize}
\item $h_{1,2}(\Gamma) = \delta_{1,2}(\Gamma)$. 
\end{itemize}
\end{theorem}

This theorem will be a particular case of a more general result for $\Theta$-Anosov subgroups of a semi-simple Lie group $G$. Recall that $C^*(\Theta)$ denotes the set of non-negative linear combinations of the simple roots $\theta \in \Theta$. Let us denote by $W(\Theta)$ the span of $\{w_\theta, \theta \in \Theta\}$, and define 
\[D^*(\Theta) = \{\phi = \alpha + \beta, \alpha \in C^*(\Theta), \beta \in W(\Theta)\} \subset \fa^*~.\]

\begin{theo} \label{theo-GeneralComparisonEntropyCriticalExponent}
Let $\Gamma$ be a $\Theta$-Anosov subgroup of $G$. Then
\begin{itemize}
    \item $h_\phi(\Gamma) \leq \delta_\phi(\Gamma)$ for all $\phi\in D^*(\Theta)\cap \fa_+^*(\Theta)$,
    \item $h_\phi(\Gamma) = \delta_\phi(\Gamma)$ for all $\phi \in W(\Theta)\cap \fa_+^*(\Theta)$.
\end{itemize}

If $\G$ is moreover Zariski dense in $G$, then
\begin{itemize}
\item $h_\phi(\G) =\delta_\phi(\G)$ for all $\phi \in D^*(\Theta)\cap \fa_+^*(\Theta)$.
\end{itemize}
\end{theo}

The conditions on $\phi$ might look exotic, but they will appear naturally in view of Corollary \ref{coro - ComparisonMuLambdaProximalElement}. 

\paragraph{{\it Main example}} Let $\G \subset \SL(n,\R)$ be a projective Anosov subgroup, which is thus $\Theta$-Anosov for $\Theta=\{\alpha_{1,2}, \alpha_{n-1,n}\}$. The fundamental weights associated to $\alpha_{1,2}$ and $\alpha_{1,n}$ are respectively $\epsilon_1$ and $\epsilon_n$. Therefore, $\alpha_{1,2}$ belongs to 
$D^*(\Theta)\cap \fa_+^*(\Theta)$ and $\alpha_{1,n}$ belongs to $W(\Theta)\cap \fa_+^*(\Theta)$. Thus Theorem \ref{theo-GeneralComparisonEntropyCriticalExponent} implies Theorem \ref{theo-specialcaseSL(n,R)}.

\begin{rmk}
The second part of Theorem \ref{theo-GeneralComparisonEntropyCriticalExponent} actually holds as soon as the Zariski closure of\,  $\G$ is semisimple (by simply restricting to the Zariski closure). A typical example where we don't know whether the equality $\delta_{1,2} = h_{1,2}$ holds is a deformation of a projective Anosov subgroup of $\SL(n,\R)$ inside $\mathrm{Aff}(\R^n) \subset \SL(n+1,\R)$.
\end{rmk}


Recall that an element $g\in \SL(n,\R)$ is called proximal if it has an attracting fixed point in $\proj(\R^n)$. Equivalently, $g$ is proximal if it satisfies $\lambda_1(g)> \lambda_2(g)$. The attracting point of $g$ is then the eigenline $L_+(g)$ for the (necessarily real) eigenvalue $\pm e^{\lambda_1(g)}$, and its basin of attraction is the complement of an invariant hyperplane $H_-(g)$. We will need a quantified version of proximality. The following definition is adapted from \cite[Definition 5.6]{guichard2012anosov}. We equip $\proj(\R^n)$ with the round metric induced by the standard scalar product of $\R^n$ and denote by $d_\proj$ the associated distance. If $L$ is a line in $\R^n$ and $H$ a linear hyperplane, we denote by $d_\proj(L,H)$ the distance between $L$ (seen as a point in $\proj (\R^n)$) and $H$ (seen as a projective hyperplane). If $H,H'$ are two projective hyperplanes, we denote by $d_\proj(H,H')$ their Hausdorff distance.

\begin{defi} \label{def - proximality}
Given $r>0$ and $0<\epsilon<1$, a matrix $g\in \SL(n,\R)$ is called $(r,\epsilon)$-proximal if it is proximal and, moreover, $g$ is $\epsilon$-Lipschitz on the ball $B(L_+(g),r)$.

If $\Theta$ is a subset of $\Delta_s$, we say that $g\in G$ is \emph{$(\Theta,r,\epsilon)$-proximal} if $\rho_\theta(g)$ is $(r,\epsilon)$-proximal for all $\theta\in \Theta$. Finally, we says that $g$ is \emph{$(r,\epsilon)$-loxodromic} if $g$ is $(\Delta_s,r,\epsilon)$-proximal.
\end{defi}

Note that, if $g$ is $(r,\epsilon)$-proximal, then $d(L_+(g),H_-(g)) >r$ and
\[\Vert \mathrm d_{L_+(g)}g\Vert \leq \epsilon~,\]
where $\mathrm d_xg$ is the derivative of $g$ at $x\in \mathbf P(\R^n)$.

We need to compare the Cartan and Jordan projections of proximal elements, this will be the purpose of Lemma \ref{lem - ComparisonMuLambdaProximalElement}. We will use the following elementary topological result: 
\begin{prop}\label{lem- compactity of loc compact transitive action}
Let $G$ be a locally compact group acting transitively on a Hausdorff space $X$ and let $x$ be a point in $X$. Then for all compact subset $M$ of $X$ there exists a compact set $M'$ of $G$ such that 
$M\subset M' \cdot x $. 
\end{prop}

\begin{proof}
Let $M_0$ be a compact neighborhood of the identity in $G$. Since $G$ acts transitively on $X$, $\bigcup_{g\in G} g\overset{\circ}{M}_0 \cdot x \supset M$. By compactness of $M$ we can extract a finite cover, $M\subset \bigcup_{i\in \{1, m\}} g_i \overset{\circ}{M}_0 \cdot x $, for some $g_i \in G$.

Then $M' = \bigcup_{i\in \{1, m\}} g_i M_0$ fulfills the conclusion of the Lemma. 
\end{proof}

\begin{lemme} \label{lem - ComparisonMuLambdaProximalElement}
Let $\theta$ be a simple root of $G$. Then for any $g\in G$, we have
\[w_\theta(\lambda(g)) \leq w_\theta(\mu(g))~.\]
Moreover, for every $r>0$, there exists $\epsilon >0$ and a constant $C$ such that, if $g$ is $(\theta,r,\epsilon)$-proximal, then
\[w_\theta(\mu(g)) \leq w_\theta(\lambda(g)) + C\]
and
\[ \theta(\mu(g)) \leq \theta(\lambda(g)) + C~.\]
\end{lemme}
Taking linear combinations of $\theta$ and $w_\theta$ for $\theta \in \Theta$, we deduce the following:
\begin{coro} \label{coro - ComparisonMuLambdaProximalElement}
    Let $\Theta$ be a subset of $\Delta_s$ and $\phi \in D^*(\Theta)$. Then for every $r>0$, there exists $\epsilon, C>0$ such that, if $g$ is $(\Theta,r,\epsilon)$-proximal, then
    \[\phi(\mu(g)) \leq \phi(\lambda(g)) + C~.\]
    If, moreover, $\phi \in W(\Theta)$, then there exists $\epsilon, C>0$ such that, if $g$ is $(\Theta,r,\epsilon)$-proximal, then
    \[\vert \phi(\mu(g)) -\phi(\lambda(g))\vert \leq C~.\]
\end{coro}

\begin{proof}[Proof of Lemma \ref{lem - ComparisonMuLambdaProximalElement}]
Taking the fundamental linear representation $\rho_\theta$, it is sufficient to prove the inequalities for $g\in \SL(n,\R)$ and $\theta = \alpha_{1,2}$. 

For the first inequality, note that
\begin{eqnarray*}
\mu_1(g) &=& \log \sup_{u\in \C^n\backslash \{0\}} \frac{\Vert g u \Vert}{\Vert u \Vert}\\
& \geq & \log \frac{\Vert g u_1 \Vert}{\Vert u_1 \Vert} = \lambda_1(g)~,
\end{eqnarray*}
where $u_1$ is an eigenvector for the eigenvalue of $g$ of highest module.

We conclude that
\begin{eqnarray*}
w_\theta(\mu(g)) &=& \mu_1(g)\\
& \geq & \lambda_1(g) = w_\theta(\lambda(g))~.
\end{eqnarray*}

Now, fix $r>0$. The subset of $M$ of $\proj(\R^n)\times \proj({\R^n}^*)$ consisting of pairs of a line $L$ and a hyperplane $H$ such that $d(L , H) \geq r$ is a compact subset of the set of pairs $(L,H)$ which are in general position. Since $\SL(n, \R) $ is locally compact and acts transitively on this latter set, by Proposition \ref{lem- compactity of loc compact transitive action}, there exists a compact set $M'\subset \SL(n,\R)$ such that $M \subset M'\cdot ([e_1], e_1^\perp)$. By compactness, there exists a constant $C$ such that the action of every $m\in M'$ on $\proj(\R^n)$ is $C$-bilipschitz. 

Choose $\epsilon < \frac{1}{C^2}$ and let $g\in \SL(n,\R)$ be $(r,\epsilon)$-proximal. Choose $m\in M'$ such that $h=m^{-1}gm$ satisfies
\[\left\{\begin{array}{l}
L_+(h) = m^{-1} L_+(g) = [e_1]\\
H_-(h) = m^{-1} \cdot H_-(g) = e_1^\perp
\end{array}
\right.\]
Then $h$ is $(\frac{r}{C}, C^2 \epsilon)$-proximal. In particular,
$\Vert \mathrm d_{h^+}(h) \Vert \leq C^2 \epsilon < 1$. 

Since $L_+(h)$ and $H_-(h)$ are orthogonal, we get that \[\mu_1 (h)= \lambda_1(h) = \lambda_1(g)\] and, denoting $\hat h$ the restriction of $h$ to $e_1^\perp$,
\[\mu_2(h) = \mu_1(\hat h) \geq \lambda_1(\hat h) = \lambda_2(h) = \lambda_2(g)~.\]
Hence
\[(\mu_1-\mu_2)(h) \leq (\lambda_1-\lambda_2)(g)~.\]

Finally, by Proposition \ref{prop - Benoist cartan projection continuous}, there is a constant $D$ (depending only on $M'$) such that
\[|\mu_1(h) - \mu_1(g)| < D\]
 and
 \[|\mu_2(h) - \mu_2(g)| < D~.\]
 It follows that $(\mu_1-\mu_2)(g)\leq (\lambda_1-\lambda_2)(g)+2D$. We also have \[\mu_1(g) \leq \mu_1(h)+D=\lambda_1(h)+D=\lambda_1(g)+D~.\]
\end{proof}


A result of Abels--Margulis--Soifer ensures that in a Zariski dense subgroup of a semisimple linear group, it is possible to make all elements $(r,\epsilon)$-loxodromic up to left multiplication by a finite set. 
\begin{lemme}[{\cite[Theorem 6.8]{abels1995semigroups}}]\label{lem-AMS}
Let $\Gamma$ be a Zariski dense subgroup of $G$. Then there exists $r>0$ such that, for any $\epsilon>0$ there is a finite subset $F$ of $\Gamma$ such that, for every $\gamma \in \Gamma$, there exists $f\in F$ such that $f\gamma$ is $(r,\epsilon)$-loxodromic.
\end{lemme}

If $\Gamma$ is not assumed to be Zariski dense, then it may not contain loxodromic elements. However, if $\Gamma$ is $\Theta$-Anosov, it certainly contains $\Theta$-proximal elements, and we have the analogous statement:

\begin{lemme}[{\cite[Theorem 5.9]{guichard2012anosov}}]\label{lem-AMS guichard wienhard}
Let $\Gamma$ be a (not necessarily Zariski dense) $\Theta$-Anosov subgroup of $G$. Then there exists $r>0$ such that, for all $\epsilon>0$, there is a finite subset $F$ of $\Gamma$ such that, for every $\gamma \in \Gamma$, there exists $f\in F$ such that $f\gamma$ is $(\Theta,r,\epsilon)$-proximal.
\end{lemme}

We will also need to control the number of $\Theta$-proximal elements in a conjugacy class.

Let $\vert \cdot \vert$ be the word length on $\Gamma$ associated to a finite set of generators and let $d_\infty$ be a distance on $\Gamma \cup \partial_\infty \Gamma$ inducing the Gromov topology. For every $[\gamma] \in [\Gamma]$, define
\[l([\gamma]) = \inf_{\gamma'\in [\gamma]} \vert \gamma'\vert~.\]
We call an element $\gamma \in \Gamma$ \emph{$\eta$-hyperbolic} if $\gamma$ has an attracting fixed point $\gamma_+$ and a repelling fixed point $\gamma_-$ in $\partial_\infty \Gamma$ such that $d_\infty (\gamma_-,\gamma_+) >\eta$. We have the following properties of $\eta$-hyperbolic elements:

\begin{prop} \label{prop-Properties hyperbolic elements}
 Let $\G$ be a Gromov hyperbolic group.   For every $\eta, \eta'>0$, there exist constants $C,C'>0$ such that
    \begin{itemize}
    \item every $\eta$-hyperbolic element $\gamma \in \Gamma$ satisfies
    \[\vert \gamma \vert \leq l([\gamma]) + C~;\]
    \item if, moreover, $\vert \gamma \vert \geq C'$, then $d_\infty(\gamma, \gamma_+) <\eta'$.
    \end{itemize}
\end{prop}

\begin{proof}
By construction of the Gromov topology of $\partial_\infty \Gamma$, there is a constant $C$ such that every bi-infinite geodesic in the Cayley graph of $\G$ with endpoints $x,y$ such that $d_\infty(x,y)>\eta$ is at distance at most $\frac{C}{2}$ from the origin $\mathbf 1_\G$ (see for instance \cite[Lemma 7.1]{CK}).

Reformulated in terms of word length, it shows the existence of a constant $C>0$ such that, for any $\eta$-hyperbolic element $\g\in \G$ there exits $m\in \G$ with $\vert m \vert \leq \frac C 2$ and $\vert m\g m^{-1}\vert = \ell([\g])$. We deduce that 
\[\vert \g \vert \leq \ell([\g])+ C~.\]

Now let $(\g_n)_{n\in\mathbb N}$ be a sequence of $\eta$-hyperbolic elements with $\vert \g_n \vert \underset{n\to +\infty}{\longrightarrow} +\infty$. Let $({\gamma_n}_-, {\gamma_n}_+)$ be a geodesic axis for $\gamma_n$, and let $[\mathbf 1_\G, {\gamma_n}_+)$ be a geodesic ray from $\mathbf 1_\G$ to ${\gamma_n}_+$. Since $d_\G(\mathbf 1_\G, ({\gamma_n}_-, {\gamma_n}_+))$ is uniformly bounded (where $d_\G$ is the word metric $d_\G(x,y)=\vert xy^{-1}\vert$ for $x,y\in \G$), so are the distances $d_\G(\gamma_n, ({\gamma_n}-, {\gamma_n}_+))$ and $d_\G([\mathbf 1_\G, {\gamma_n}_+), ({\gamma_n}_-, {\gamma_n}_+))$. We deduce that $d_\G(\gamma_n, [\mathbf 1_\G, {\gamma_n}_+))$ is uniformly bounded and, since $\vert \gamma_n \vert \underset{n\to +\infty}{\longrightarrow} +\infty$, we conclude that \[d_\infty(\gamma_n, {\gamma_n}_+)\underset{n\to +\infty}{\longrightarrow} 0~.\]
This convergence is uniform in $\vert \gamma \vert$ by compactness of $\Gamma \cup \partial_\infty \Gamma$. 
\end{proof}

The following proposition will allow us to reduce the counting of $(\Theta,r,\epsilon)$-proximal elements to a counting intrinsic to Gromov hyperbolic groups. 

\begin{prop} \label{prop - Proximal VS Hyperbolic}
Let $\Gamma$ be a $\Theta$-Anosov subgroup of $G$. For all $\eta,\epsilon>0$, there exists $r>0$ and a constant $D>0$ such that, if $\gamma \in \Gamma$ is $\eta$-hyperbolic and $\vert \gamma \vert \geq D$, then $\gamma$ is $(\Theta,r,\epsilon)$-proximal.

Conversely, for all $r>0$, there exists $\eta>0$ such that every $(\Theta,r,\epsilon)$-proximal element $\gamma$ (for arbitrary $\epsilon$) is $\eta$-hyperbolic.
\end{prop}

\begin{proof}
By Corollary \ref{coro - anosov is projective anosov}, one can assume that $\Gamma$ is a projective Anosov subgroup of $\SL(n,\R)$. Let $\xi:\partial_\infty \Gamma \to \proj(\R^n)$ and $\xi^*: \partial_\infty \Gamma \to \proj({\R^n}^*)$ be the corresponding boundary maps.

Fix $\eta$ and $\epsilon>0$. The set $ B_1 =\{(x,y)\in \partial_\infty \Gamma \times \partial_\infty \Gamma \mid d_\infty(x,y) \geq \eta\}$ is compact and does not intersect the diagonal. By continuity and transversality of the map $(\xi,\xi^*)$, there exists $r>0$ such that for all $(x,y)\in B_1$, we have $d_\proj(\xi(x), \xi^*(y))> 3r$. In particular, $d_\proj(\xi(\g_+), \xi^*(\g_-)) \geq 3r$ for every $\eta$-hyperbolic element $\g$.

For $g\in \SL(n,\R)$ with $\mu_1(g) > \mu_2(g)$, write its Cartan decomposition $g= k_1 e^{\mu(g)} k_2$ and set $V_-(g) = k_2^{-1}e_1^\perp$. One can show the existence of a constant $D_1$ (depending only on $r$) such that the projective action of $g$ is $D_1 e^{\mu_2(g)-\mu_1(g)}$-Lipschitz on the set $\{v \in \proj(\R^n) \mid d_\proj(v, V_-(g)) \geq r\}$.

Rewriting the dynamics preserving property of the boundary maps $\xi$ and $\xi^*$ (see Theorem \ref{theo - AnosovLimitMaps}), we get the existence of some $\eta'>0$ (depending on $r$) such that, if $\g \in \G$ and $x\in \partial_\infty \G$ satisfy $d_\infty (g^{-1},x)\leq \eta'$, then $d_\proj(V_-(g), \xi^*(x)) < r$. By Proposition \ref{prop-Properties hyperbolic elements}, there exists a constant $D_2$ such that every $\eta$-hyperbolic element $\gamma \in \Gamma$ with $\vert \gamma \vert \geq D_2$ satisfies $d_\infty(\gamma^{-1},\gamma_-) <\eta'$.

Let $(\alpha, A)$ be such that $(\mu_1-\mu_2)(\g)\geq \alpha \vert \g \vert -A$ for all $\g\in\G$. Finally, choose $D = \max(D_2,\frac{-\log(\epsilon)+A+\log(D_1)}{\alpha})$. Let $\gamma$ be $\eta$-hyperbolic with $\vert \gamma \vert \geq D$. Then:
\begin{itemize}
\item $d_\proj(\xi(\gamma_+),\xi^*(\gamma_-)) > 3r$ by construction of $r$,

\item $D_1 e^{(\mu_2-\mu_1)(\gamma)} \leq D_1 e^{A-\alpha D}\leq \epsilon$, which means that $\gamma$ is $\epsilon$-Lipschitz on the set $\{v\mid d_\proj (v, V_-(\gamma))\geq r\},$
\item $d_\infty(\gamma, \gamma_+) \leq \eta'$ since $D\geq D_2$, hence $d_\proj(V_-(\gamma), \xi^*(\gamma_-)) \leq r$, which shows that $d_\proj(\xi(\gamma_+), V_-(\gamma))>2r$ and $\{v\mid d_\proj(v, V_-(\g))\geq r\} \supset B(\xi(\gamma_+),r)$.
\end{itemize}
We conclude that $\gamma$ is $(r,\epsilon)$-proximal.

Conversely, fix $r>0$. Then the set $B_2 = \{(x,y)\in \partial_\infty \Gamma \times \partial_\infty \Gamma \mid d_\proj(\xi(x),\xi^*(y)) \geq r\}$ is compact by continuity of $(\xi,\xi^*)$ and does not intersect the diagonal (since $\xi(x)\in \xi^*(x)$ for all $x$). Hence there exists $\eta>0$ such that $d_\infty(x,y)\geq \eta$ for all $(x,y)\in B$. In particular, every $(r,\epsilon)$-proximal element is $\eta$-hyperbolic.
\end{proof}

Now, the following lemma controls the number of $\eta$-hyperbolic elements in a given conjugacy class:

\begin{lemme}[See {\cite[Paragraph 7]{CK}}] \label{lem - HyperbolicConjugacies}
There exists $\eta_0>0$ such that every conjugacy class of an infinite order element $[\g]\in [\G]$ contains at least one $\eta_0$-hyperbolic element. 

Moreover, for every $\eta>0$, there exists a constant $C$ such that any conjugacy class $[\g]$ of infinite order contains at most $C l([\g])$ elements that are $\eta$-hyperbolic.
\end{lemme}

\begin{coro} \label{lem - ProximalConjugacyClass}
Let $\Gamma$ be a $\Theta$-Anosov subgroup of $G$. Then there exists $r>0$ such that, for any $\epsilon>0$, there exists $D>0$ such that every conjugacy class of an infinite order element $[\g]\in [\Gamma]$ with $l([\g]) \geq D$ contains at least one $(\Theta,r,\epsilon)$-proximal element.

Moreover, for any $r,\epsilon>0$, there exists a constant $C$ such that any conjugacy class $[\g]$ of an infinite order element contains at most $C l([\g])$ elements that are $(\Theta,r,\epsilon)$-proximal.
\end{coro}

\begin{proof}

Let $\eta_0$ be such that Lemma \ref{lem - HyperbolicConjugacies} applies. By Proposition \ref{prop - Proximal VS Hyperbolic},  there exists $r>0$ such that, for every $\epsilon>0$ there is a $D>0$ such that $\g\in \G$ is $(\Theta,r,\epsilon)$-proximal whenever $\g$ is $\eta_0$-hyperbolic and $\vert \g \vert \geq D$. Let $[\g]$ be the conjugacy class of an infinite order element such that $l([\g]) \geq D$. By Lemma \ref{lem - HyperbolicConjugacies}, $[\g]$ contains an $\eta_0$-hyperbolic element $\g$ which satisfies $\vert \g \vert \geq l([\g])\geq D$. By construction of $D$ and $r$, this element $\g$ is $(\Theta,r,\epsilon)$-proximal. 

Conversely, for any $r,\epsilon>0$, by Proposition \ref{prop - Proximal VS Hyperbolic} there exists $\eta >0$ such that $\g$ is $\eta$-hyperbolic whenever $\g$ is $(\Theta,r,\epsilon)$-proximal. We thus get
\begin{align*}\Card \{\g' \in [\g]\mid \g' \textrm{ $(\Theta,r,\epsilon)$-proximal}\} &\leq \Card \{\g' \in [\g] \mid \g' \textrm{ $\eta$-hyperbolic}\} \\
& \leq C l([\g])
\end{align*}
where the constant $C$ is given by Lemma \ref{lem - HyperbolicConjugacies}.
\end{proof}

We now have all the tools to prove Theorem \ref{theo-GeneralComparisonEntropyCriticalExponent}.

\begin{proof}[Proof of Theorem \ref{theo-GeneralComparisonEntropyCriticalExponent}]

Let $\varphi$ be an element of $D^*(\Theta)\cap \fa_+^*(\Theta)$. Since $\Gamma$ is $\Theta$-Anosov and $\phi \in \fa_+^*(\Theta)$, there exist $\alpha, A>0$ such that for all $\g \in \G$, 
\[\alpha \vert \g \vert - A \leq \phi(\mu(\g)) \leq \frac{1}{\alpha}\vert \g \vert + A\]
and
\[\alpha l([\g]) - A \leq \phi(\lambda(\g)) \leq \frac{1}{\alpha}l([\g]) + A~.\\ \]

We first prove the inequality $h_\varphi \leq \delta_\varphi$, then the reverse inequality.


Fix $r>0$ given by Corollary \ref{lem - ProximalConjugacyClass}, then $\epsilon>0$ given by Corollary \ref{coro - ComparisonMuLambdaProximalElement}, then $D$ as in \ref{lem - ProximalConjugacyClass}. With these choices, every conjugacy class $[\gamma]$ with $\phi(\lambda(\g)) \geq \frac{D-A}{\alpha}$ contains at least one $(\Theta,r,\epsilon)$-proximal element.

Define
\[\Gamma_\phi(R) =\{\g \in \G \mid \phi(\mu(\gamma)) \leq R\}~.\]
For every $R>0$ and every conjugacy class $[\g] \in [\G]$ of infinite order such that $\frac{D-A}{\alpha}\leq \varphi(\lambda(\g)) \leq R$, we can find $\g' \in [\g]$ which is $(\Theta,r,\epsilon)$-proximal by Corollary~\ref{lem - ProximalConjugacyClass}. By Corollary \ref{coro - ComparisonMuLambdaProximalElement}, since $\varphi \in D^*(\Theta)$, there exists a constant $C$ (independent of $R$) such that $\gamma' \in \Gamma_\phi(R+C)$. We deduce that
\[\Card \Gamma_\phi(R+C) \geq \Card \{[\g] \in [\Gamma] \textrm{ of infinite order } \mid \frac{D-A}{\alpha}\leq \varphi(\lambda(\g))\leq R\}~.\]
The inequality $\delta_\varphi \geq h_\varphi$ easily follows.


To show $\delta_\varphi \leq h_\varphi$, assume first that $\G$ is Zariski dense. Take $r>0$ given by Abels--Margulis--Soifer's Lemma \ref{lem-AMS}, then $\epsilon$ such that Lemma \ref{coro - ComparisonMuLambdaProximalElement} applies to $(r,\epsilon)$-loxodromic elements. Lemma \ref{lem-AMS} then gives us a finite subset $F\subset \G $ such that for any $\g\in \G$, there exists $f\in F$ such that $f\g$ is $(r,\epsilon)$-loxodromic. By Proposition \ref{prop - Benoist cartan projection continuous} and Corollary \ref{coro - ComparisonMuLambdaProximalElement}, there is a constant $C'$ depending on $F, r$ and $\epsilon$ such that $|\varphi(\lambda(f\g)) - \varphi(\mu(\g))|\leq C'$. We thus have
\[
\Card\Gamma_\phi(R) \leq \Card(F) \cdot \Card\{ \g \in \Gamma \mid \phi(\lambda(\g)) \leq R+C', \textrm{ $\g$ is $\epsilon$-loxodromic}\}.\]
Now, by Corollary \ref{lem - ProximalConjugacyClass}, every conjugacy class $[\g]$ such that $\phi(\lambda(\gamma)) \leq R+C'$ contains at most $\frac{C(R+A+C')}{\alpha}$ elements that are $(r,\epsilon)$-loxodromic, for the constant $C$ given by Corollary \ref{lem - ProximalConjugacyClass}.

We thus obtain
\[\Card \Gamma_\phi(R) \leq \frac{C(R+A+C')}{\alpha} \Card F \cdot \Card \{[\g]\in [\Gamma] \mid \varphi(\lambda(\g))\leq R+C\}~,\]
from which the inequality $\delta_\varphi \leq h_\varphi$ easily follows.

If $\Gamma$ is not assumed Zariski dense, we can still apply Lemma \ref{lem-AMS guichard wienhard} and choose $r>0$ for which every element $\g\in \G$ can be made $(\Theta,r,\epsilon)$-proximal by multiplying it by an element $f$ in a finite set. By Corollary \ref{coro - ComparisonMuLambdaProximalElement}, we do have $|\phi(\lambda(f \g)) - \phi(\mu(\g))| \leq C'$ provided that $\phi$ belongs to $W(\Theta)$. The rest of the proof works in a similar way and we eventually obtain the inequality $\delta_\phi \leq h_\phi$ for $\phi \in W(\Theta)$.
\end{proof}

\subsubsection{Critical exponent and tensor product}

We conclude this section by discussing the behaviour of the simple root and highest root critical exponents under taking tensor products of representations. 

Given $g\otimes h$ a pure tensor in $\SL(n_1,\R)\otimes \SL(n_2,\R) \subset \SL(n_1n_2,\R)$, one easily verifies the following identities:
\begin{equation} \label{eq - Highest weight tensor product}
\mu_1(g\otimes h) - \mu_{n_1n_2}(g\otimes h) = \mu_1(g) - \mu_{n_1}(g) + \mu_1(h) - \mu_{n_2}(h)~,\end{equation} 
\begin{equation} \label{eq - Simple weight tensor product}
\mu_1(g\otimes h) - \mu_{2}(g\otimes h) = \inf \{\mu_1(g) - \mu_2(g), \mu_1(h) - \mu_{2}(h)\}~.\end{equation}

Now, let $\Gamma$ be a finitely generated group and $\rho_1: \Gamma \to \SL(n_1,\R)$ and $\rho_2: \Gamma \to \SL(n_2,\R)$ two discrete and faithful representations. From \eqref{eq - Highest weight tensor product}, one obtains the following

\begin{prop} \label{prop - Critical exp HW tensor product}
The highest weight critical exponent of $\rho_1\otimes \rho_2$ satisfies
\[\delta_{1,n_1n_2}(\rho_1\otimes \rho_2)\leq \frac{\delta_{1,n_1}(\rho_1) \delta_{1,n_2}(\rho_2)}{\delta_{1,n_1}(\rho_1) +\delta_{1,n_2}(\rho_2)}~.\]
\end{prop}

\begin{proof}
Take $s> \frac{\delta_{1,n_1}(\rho_1) \delta_{1,n_2}(\rho_2)}{\delta_{1,n_1}(\rho_1) +\delta_{1,n_2}(\rho_2)}$. Then 
\[\frac{s}{\delta_{1,n_1}(\rho_1)}+ \frac{s}{\delta_{1,n_2}(\rho_2)} > 1~.\]
Hence there exists $p,q>1$ such that
\[\frac{s}{\delta_{1,n_1}(\rho_1)} > \frac{1}{p}~,~\frac{s}{\delta_{1,n_2}(\rho_2)} > \frac{1}{q}~\textrm{and}~ \frac{1}{p}+\frac{1}{q}=1~.\]

Now,
\begin{eqnarray*}
& & \sum_{\gamma\in \G} e^{-s \mu_1(\rho_1\otimes \rho_2(\gamma)) - \mu_{n_1n_2}(\rho_1\otimes \rho_2(\gamma))}\\
&=& \sum_{\gamma \in \Gamma} e^{-s \mu_{n_1}(\rho_1(\gamma))} e^{-s \mu_{n_2}(\rho_2(\gamma))} \quad \textrm{by \eqref{eq - Highest weight tensor product}} \\
&\leq & \left(\sum_{\gamma\in \Gamma} e^{-ps \mu_{n_1}(\rho_1(\gamma))}\right)^{\frac{1}{p}}\left(\sum_{\gamma\in \Gamma} e^{-qs \mu_{n_2}(\rho_2(\gamma))}\right)^{\frac{1}{q}} \quad \textrm{By H\"older's inequality}\\
& < & +\infty \quad \textrm{ since $ps> \delta_{1,n_1}(\rho_1)$ and $qs > \delta_{1,n_2}(\rho_2)$.}
\end{eqnarray*}
Hence $s> \delta_{1,n_1n_2}(\rho_1\otimes \rho_2)$.
\end{proof}

\begin{rmk}
 Proposition \ref{prop - Critical exp HW tensor product} is sharp: it is an equality when the highest weight length spectra of $\rho_1$ and $\rho_2$ are proportional, which happens for instance when $\rho_1$ and $\rho_2$ are tensor powers of the same linear representation.
\end{rmk}

\begin{prop} \label{prop - Critical exp SW tensor product}
The simple weight critical exponent of $\rho_1\otimes \rho_2$ satisfies
\[\delta_{1,2}(\rho_1\otimes \rho_2) = \max \{\delta_{1,2}(\rho_1), \delta_{1,2}(\rho_2)\}~.\]
\end{prop}

\begin{proof}
By \eqref{eq - Simple weight tensor product}, we have
\[\Card \{\gamma \in \G \mid \mu_1(\rho_i(\gamma))-\mu_2(\rho_i(\gamma)) \leq R\} \leq \Card \{\g \in \G \mid (\mu_1 - \mu_2)(\rho_1\otimes \rho_2(\g)) \leq R\}\]
for $i=1,2$, from which we deduce that
\[\delta_{1,2}(\rho_i) \leq \delta_{1,2}(\rho_1\otimes \rho_2)\] for $i= 1,2$. On the other hand, we have
\begin{align*}\Card \{\g \in \G \mid (\mu_1 - \mu_2)(\rho_1\otimes \rho_2(\g)) \leq R\} &\leq \Card \{\g \in \G \mid (\mu_1-\mu_2)(\rho_1(\g)) \leq R\}\\ &~~~~ + \Card \{\g \in \G \mid (\mu_1-\mu_2)(\rho_2(\g)) \leq R\}~ \end{align*}
from which we get
\[\delta_{1,2}(\rho_1\otimes \rho_2(\Gamma)) \leq \max \{\delta_{1,2}(\rho_1), \delta_{1,2}(\rho_2)\}~.\]
\end{proof}

Applying this to the tensor product of two Fuchsian representations, we obtain Theorem \ref{theo - simple root exponent tensor product Fuchsian}:

\begin{proof}[Proof of Theorem \ref{theo - simple root exponent tensor product Fuchsian} (Assuming Corollary \ref{coro-NotConvexCocompact})]
Let $\Gamma$ be the fundamental group of a closed surface and $j_1$ and $j_2$ be two Fuchsian representations of $\Gamma$ into $\SL(2,\R)$.

Since the boundary at infinity of $\Gamma$ is a topological circle, the symmetrized limit set $\Lambda_{j_1\otimes j_2}^{\mathrm{sym}}$ of $j_1\otimes j_2(\Gamma)$ in $\proj(\R^4)\times \proj({\R^4}^*)$ has Hausdorff dimension at least~$1$.

Now, it is well-known\footnote{As a consequence for instance of the volume growth of balls in the hyperbolic plane.} that $\delta_{1,2}(j_1) = \delta_{1,2}(j_2) = 1$. Applying Proposition \ref{prop - Critical exp SW tensor product}, we obtain
\[\delta_{1,2}(j_1\otimes j_2) =1~.\]
Using the right inequality in Corollary \ref{coro-NotConvexCocompact}, we conclude that
\[\delta_{1,2}(j_1\otimes j_2) = \DimH(\Lambda_{j_1\otimes j_2}^{\mathrm{sym}}) = 1~.\]
\end{proof}

\subsection{Projectively convex cocompact representations}\label{subsec - projectively cc representations}

We now recall some background on projective convex geometry and its relation to Anosov groups. We refer to \cite{DGK} and \cite{zimmer} for more details.

\subsubsection{Hilbert geometries}

Let us first recall some classical facts on convex subsets of $\proj(\R^n)$ and their Hilbert geometry. The main references for this are Benoist \cite{Benoist2001convexe}, Crampon  \cite{crampon2009entropies,crampon2011dynamics}, Danciger--Guéritaud--Kassel \cite{DGK}. 

An open domain $\Omega$ of $\proj(\R^n)$ is said to be \emph{properly convex} if it is convex and bounded in some affine chart. Hilbert constructed a natural projective invariant distance on a properly convex domain in $\proj(\R^n)$. To define it, let us choose an affine chart in which $\Omega$ is bounded. Given $u$ and $v$ two points in this affine chart, we denote by $uv$ the length of the segment $[u,v]$.

\begin{defi}
Let $x,y$ be two points in $\Omega$. Let $a$ and $b$ denote respectively the intersections of the half lines $[y,x)$ and $[x,y)$ with $\partial \Omega$. Then the Hilbert distance between $x$ and $y$ is given by
\[d_\Omega(x,y) = \frac{1}{2}\log\left(\frac{xb\cdot ay}{ax \cdot yb}\right)~.\]
\end{defi}

This distance actually does not depend on the chosen affine chart (it is essentially the logarithm of a projective cross-ratio), and if a projective transformation maps $\Omega$ to $\Omega'$, then it induces an isometry between the Hilbert distances of $\Omega$ and $\Omega'$. In particular, the group of projective transformations preserving $\Omega$ acts by isometries for the Hilbert distance.

The Hilbert distance is induced by a Finsler metric for which straight lines are geodesic. We will say that a proper convex domain $\Omega$ is \emph{Gromov hyperbolic} if $(\Omega,d_\Omega)$ is a hyperbolic metric space in the sense of Gromov. This implies in particular that $\Omega$ is strictly convex and has $\mathcal C^1$ boundary. Benoist  \cite{Benoist2001convexe} gave more precise characterizations of Gromov hyperbolic convex sets.

\begin{example} \label{ex:Ellipsoid}
Let $\Omega \subset \proj(\R^n)$ be the set of lines in $\R^n$ in restriction to which the quadratic form
\[\mathbf q :(x_1,\ldots,x_n) \mapsto x_1^2+ \ldots + x_{n-1}^2 - x_n^2\]
is negative. The convex set $\Omega$ identifies with the symmetric space of the group $\SO(n-1,1)$ of linear transformations preserving $\mathbf q$, that is, the hyperbolic space of dimension $n-1$. In that case, the Hilbert distance on $\Omega$ is induced by the $\SO(n-1,1)$-invariant Riemannian metric of constant curvature $-1$. In particular, $\Omega$ is a Gromov hyperbolic convex domain.
\end{example}

\begin{example} \label{ex:ConvexSymmetricSpace}
Let $\Sym^2({\R^k}^*)$ be the space of quadratic forms on $\R^k$ and let $\Omega \subset \proj(\Sym^2({\R^k}^*))$ be the projectivization of the cone of positive definite quadratic forms. Then the group $\SL(k,\R)$ acts transitively on $\Omega$, and $\Omega$ identifies with the symmetric space $\SL(k,\R)/\SO(k)$. In that case, the Hilbert distance on $\Omega$ is related to the Cartan projection in the following way:
\[d_\Omega(x,y) = \epsilon_1(\mu(x,y)) - \epsilon_n(\mu(x,y))~.\]
Note that, for $k\geq 3$, $\Omega$ is not Gromov hyperbolic.
\end{example}

\begin{defi}\label{def - convex cocompact}
Let $\Gamma$ be a discrete subgroup of $\SL(n,\R)$ preserving a proper convex domain $\Omega\subset \proj(\R^n)$. 

We say that $\Gamma$ acts \emph{convex-cocompactly}\footnote{As explained extensively in \cite{DGK}, this naïve notion of projective convex-cocompactness is only robust when $\Omega$ is Gromov hyperbolic, which is our only case of interest here.} on $\Omega$ if there exists a non-empty $\Gamma$-invariant convex subset $C\subset \Omega$ such that $\Gamma$ acts properly discontinuously and cocompactly on $C$. 

We say that $\Gamma$ is \emph{strongly projectively convex-cocompact} if  there is a  Gromov hyperbolic convex domain $\Omega$ on which it acts convex-cocompactly.
\end{defi}

\begin{rmk}
The adjective ``strongly'' is here to distinguish the notion from a weaker notion of convex-cocompactness that includes discrete subgroups that are not hyperbolic.
\end{rmk}

\begin{example}
If $\Gamma \subset \SO(n-1,1)$ is a convex-cocompact group of hyperbolic isometries, then it preserves the convex domain $\Omega \simeq \Hyp^{n-1}$ introduced in Example \ref{ex:Ellipsoid} and acts properly discontinuously on its convex core $C \subset \Omega$. It is thus strongly projectively convex-cocompact.
\end{example}

\begin{theorem}\cite{DGK}
Let $\G$ be a discrete subgroup of $\SL(n,\R)$. If $\G$ is strongly projectively convex-cocompact, then $\G$ is projective Anosov.
\end{theorem}

More precisely, we have the following description of the boundary maps $\xi$ and $\xi^*$ associated to $\Gamma$:

\begin{theorem}\cite{DGK} \label{theo - BoundaryMapsConvexCocompact}
Let $\Gamma$ be a discrete subgroup of $\SL(n,\R)$ preserving a Gromov hyperbolic convex domain $\Omega$ and acting properly discontinuously and cocompactly on a non-empty convex set $C\subset \Omega$. Let $\overline{C}$ denote the closure of $C$ in $\proj(\R^n)$. Then $\Gamma$ is projective Anosov and
\begin{itemize}
    \item the boundary map $\xi$ is a homeomorphism from $\partial_\infty \Gamma$ to $\overline{C} \cap \partial \Omega$,
    \item for every $x\in \partial_\infty \Gamma$, $\xi^*(x)$ is the hyperplane tangent to $\partial \Omega$ at $\xi(x)$.
\end{itemize}
\end{theorem}

Conversely, Danciger--Gu\'eritaud--Kassel prove that a projective Anosov subgroup of $\SL(n,\R)$ is strongly projectively convex-cocompact as soon as it preserves a proper convex domain. In particular, we have the following:

\begin{theorem} \cite{DGK} \label{th - projective anosov are convex cocompact}
Let $I: \SL(n,\R) \to \SL(\Sym^2({\R^n}^*))$ be the representation given by the action of $\SL(n,\R)$ on the space of quadratic forms on $\R^n$. If $\Gamma \subset \SL(n,\R)$ is projective Anosov, then $I(\Gamma) \subset \SL(\Sym^2({\R^n}^*))$ is strongly projectively convex cocompact.
\end{theorem}

\subsubsection{Symmetric boundary of a divisible convex set}

A particular case of strongly projectively convex-cocompact group in $\SL(n,\R)$ is a group acting properly discontinuously and cocompactly on a Gromov-hyperbolic convex set. Crampon proved that the Hilbert critical exponent of such a group is at most $n-2$. we prove here that the Hausdorff dimension of $\Lambda_\Gamma^{sym}$ equals $n-2$, so that our main theorem does recover Crampon's result.

If $\Gamma$ acts properly discontinuously and cocompactly on a Gromov hyperbolic convex domain $\Omega$, then we have $\Lambda_\Gamma = \partial \Omega$, $\Lambda_\Gamma^* = \partial \Omega^*$ (the set of hyperplanes tangent to $\partial\Omega$) and $\Lambda^{sym}_\Gamma \subset \Lambda_\Gamma \times \Lambda_\Gamma^*$ is the graph of the homeomorphism mapping $x\in \partial\Omega$ to $x^* \overset{\textrm{def}}{=} T_x \partial \Omega$.

More generally, we prove the following:
\begin{theorem} \label{theo - HD Symmetric boundary convex}
Let $\Omega$ be a strictly convex domain of $\proj(\R^n)$ with $\mathcal C^1$ boundary. Then
\[\partial \Omega^{sym} \overset{\mathrm{def}}{=} \{(x,x^*), x\in \partial \Omega\} \subset \partial \Omega \times \partial \Omega^*\]
has Hausdorff dimension $n-2$.
\end{theorem}

Note that, while $\partial \Omega$ and $\partial\Omega^*$ are $\mathcal C^1$ spheres of dimension $n-2$, the homeomorphism $x\mapsto x^*$ is not $\mathcal C^1$ unless $\partial \Omega$ is $\mathcal C^2$ (which is very constraining for divisble convex sets). It is thus not immediate that its graph has Hausdorff dimension $n-2$ in general.

\begin{proof}[Proof of Theorem \ref{theo - HD Symmetric boundary convex}]
Fix two distinct points $x_0,x_\infty\in \partial \Omega$, and consider vectors $v_0\in x_0$, $v_\infty\in x_\infty$ and linear forms $\alpha_0\in x_0^*$, $\alpha_\infty\in x_\infty^*$ such that $\alpha_0(v_\infty)=\alpha_\infty(v_0)=1$. Let $V=x_0^*\cap x_\infty^*\subset \R^n$, so that  the affine patch $\proj(\R^n)\setminus x_\infty^*$ can be identified with $V\times \R$ through the map $(v,t)\mapsto [v+tv_\infty+v_0]$.

Since $\Omega$ is strictly convex and has $\mathcal C^1$ boundary, under this identification $\partial\Omega\setminus \{x_\infty\}$ is the graph of a $\mathcal C^1$ strictly convex function $f:V\to \R$.

Let $p_{0,\infty}:\R^n\to V$ be the projection with kernel $x_0\oplus x_\infty$, so that we also get an identification of the affine patch $\proj\left((\R^n)^*\right)\setminus x_0^\perp$ with $V^*\times \R$ through the map $(\alpha,t)\mapsto [\alpha\circ p_{0,\infty} + t \alpha_\infty + \alpha_0]$.

Under this identification, the dual point of $(u,f(u))\in \partial \Omega\setminus \{x_\infty\}$ is $(u,f(u))^*=(-d_uf,d_uf(u)-f(u))\in \partial\Omega^*\setminus \{x_\infty^*\}$. It follows that the map
\[\begin{matrix} \partial\Omega^*\setminus \{x_\infty^*\} & \to & V^* \\ (u,f(u))^* & \mapsto & -d_uf \end{matrix}\]
 is just  the projection on the first coordinate of a $\mathcal C^1$ submanifold of $V^*\times \R$, so it is $\mathcal C^1$ (even though it is the composition of the non $\mathcal C^1$ maps $x^*\mapsto x$ and $v\mapsto d_vf$).

%

The projection
\[
\begin{array}{rccc}
\pi: & \left(\partial \Omega \setminus \{x_\infty\}\right)\times \left(\partial \Omega^*\setminus \{x_\infty^*\}\right) & \to & V\times V^*\\
& \left((u, f(u)), (v,f(v))^*\right) & \mapsto & (u, \mathrm d_v f)
\end{array}
\]
is a $\mathcal C^1$ isomorphism. It is thus enough to prove that
\[\pi(\partial \Omega^{sym} \setminus \{(y,y^*)\}) = \{(u, \mathrm d_u f), u\in V\}\]
has Hausdorff dimension $d-2$.

Now let us equip $V \times V^*$ with symmetric bilinear form
\[\langle (u,\alpha), (v,\beta)\rangle = \alpha(v) + \beta(u)~.\]
This form is non-degenerate of signature $(n-2,n-2)$, and we claim that $\pi(\partial \Omega^{sym} \setminus \{(y,y^*)\})$ is a spacelike topological submanifold of dimension $n-2$, hence a Lipschitz manifold. 

Indeed, for every $u\neq v \in V$, we have
\begin{eqnarray*}
\langle (v-u, \mathrm d_v f - \mathrm d_u f), (v-u, \mathrm d_v f - \mathrm d_u f)\rangle &=& -2(\mathrm d_u f(v-u) + \mathrm d_v f(u-v))~.
\end{eqnarray*}
By strict convexity of $f$, we have
\[\mathrm d_u f(v-u) < f(v) - f(u)\]
and \[\mathrm d_v f(u-v) < f(u) - f(v)~,\]
from which we get
\[\langle (v-u, \mathrm d_v f - \mathrm d_u f), (v-u, \mathrm d_v f - \mathrm d_u f)\rangle > 0~.\]

Now, let $W$ be a spacelike subspace of $V\times V^*$ of dimension $n-2$. The previous computation implies that the orthogonal projection of $\pi(\partial \Omega^{sym} \setminus \{(x_\infty, x_\infty^*)\})$ to $W$ is injective. Since $\pi(\partial \Omega^{sym} \setminus \{(x_\infty, x_\infty^*)\})$ is a $\mathcal C^1$ manifold of dimension $n-2$, this projection is open (by Brouwer's invariance of domain), and $\pi(\partial \Omega^{sym} \setminus \{(x_\infty, x_\infty^*)\})$ is the graph of a function from an open domain of $W$ to $W^\perp$. Finally, since $\pi(\partial \Omega^{sym} \setminus \{(x_\infty, x_\infty^*)\})$ is spacelike, this graph is $1$-Lipschitz (when $W$ is endowed with $\langle \cdot , \cdot \rangle$ and $W^\perp$ with $-\langle \cdot , \cdot \rangle$). We thus conclude that $\pi(\partial \Omega^{sym} \setminus \{(x_\infty, x_\infty^*)\})$ is a Lipschitz manifold of dimension $n-2$, hence $\partial \Omega^{sym}$ has Hausdorff dimension $n-2$.
\end{proof}

\subsubsection{Hilbert entropy and critical exponent}

Let $\G$ be a discrete subgroup of $\SL(n,\R)$ preserving a proper convex subset $\Omega\subset \proj(\R^n)$. Then $\G$ is a subgroup of isometries of $(\Omega,d_\Omega)$ where $d_\Omega$ is the Hilbert distance on $\Omega$. We denote the critical exponent associated to this metric by $\delta_\Omega(\G)$:
$$\delta_\Omega(\G)\overset{\mathrm{def}}{=} \limsup_{R\tv \infty} \frac{1}{R} \log \Card \{ \g \in \G \, |\, d_\Omega(x,\g x) \leq R\}.$$




Furthermore, if $\Omega$ is strictly convex then there is a one-to-one correspondence between the set of  conjugacy classes $ [\G]$ and the closed geodesics of $\Omega/\G$. For a conjugacy class $[\g]\in[\G] $ we denote by $\ell_\Omega(\g)$ the length of the corresponding closed geodesic for the Hilbert metric on $\Omega/\G$.  The exponential growth of the number of closed geodesics is denoted by $h_\Omega(\G)$:
\[h_\Omega(\G) \overset{\mathrm{def}}{=} \limsup_{R\tv \infty} \frac{1}{R} \log \Card \{ [\g] \in [\G] \, |\, \ell_\Omega(\g)\leq R\}~.\]

The work of Coornaert--Knieper on growth rate of conjugacy classes in Gromov hyperbolic groups has the following consequence: 
\begin{theorem}[Coornaert -- Knieper, \cite{CK}]\label{th - delta = h for convex cocompact gromov hyperbolic}
Let $\G$ be acting convex-cocompactly on a Gromov hyperbolic   convex domain $\Omega\subset \proj(\R^n)$. Then:
\[\delta_\Omega(\G)=h_\Omega(\G)~.\]
\end{theorem}

For any element $\g\in \G$ we can compute the length of the closed geodesic corresponding to $[\g]$ in the quotient manifold. A direct computation shows that $\g$ acts as a translation on the geodesic joining $\g^-$ to $\g^+$ with translation distance given by $\ell_{\Omega}(\g) := \frac{1}{2}(\lambda_1(\g)-\lambda_n(\g))$, hence $h_\Omega(\G)=2h_{1,n}(\G)$.  Combining this with Theorem \ref{th - delta = h for convex cocompact gromov hyperbolic} and Theorem \ref{theo-specialcaseSL(n,R)}, we obtain
\begin{coro} \label{cor:deltaH=2delta1n}
Let $\G \subset \SL(n,\R)$ be acting convex-cocompactly on a Gromov-hyperbolic convex domain $\Omega\subset \proj(\R^n)$. Then
$$\delta_\Omega(\G)= h_\Omega(\G) = 2 h_{1,n}(\G) = 2 \delta_{1,n}(\G)~.$$ 
\end{coro}

\begin{rmk}
If $\Gamma$ is a projective Anosov subgroup of $\SL(n, \R)$, then $I(\Gamma)$ is a strongly projectively convex cocompact  subgroup of $\SL(\Sym^2({\R^n}^*)$ by Theorem \ref{th - projective anosov are convex cocompact}. One easily verifies that
\[\alpha_{1, \frac{n(n+1)}{2}}(\mu(I(g))) = 2 \alpha_{1,n}(\mu(g))\]
for all $g\in \SL(n,\R)$, hence
\[\delta_{1, \frac{n(n+1)}{2}}(I(\Gamma)) = \frac{1}{2} \delta_{1,n}(\Gamma)~,\]
as well as 
\[h_{1, \frac{n(n+1)}{2}}(I(\Gamma)) = \frac{1}{2} h_{1,n}(\Gamma)~.\]
 In this case, the Hilbert length of the closed geodesic corresponding to $I(\g)$ is given by $(\lambda_1-\lambda_n)(\g)$, and therefore $h_\Omega(I(\G)) = h_{1,n}(\G)$.
\end{rmk}



\section{Lower bound} \label{sec - lower bound}
This section is devoted to the proof of the following lower bound on the Hausdorff dimension. For a projective Anosov subgroup $\G \subset \SL(n,\R)$, we write
\[\Lambda_\G = \xi(\partial_\infty(\G)) \subset \proj(\R^n)~,\]
\[\Lambda^*_\G = \xi^*(\partial_\infty(\G)) \subset \proj({\R^n}^*)\]
and
\[\Lambda^{\mathrm{sym}}(\G) = (\xi,\xi^*)(\partial_\infty \G) \subset \proj(\R^n)\times \proj({\R^n}^*)~.\]

\begin{theorem} \label{th - lower bound}
Let $\G$ be a strongly projectively convex-cocompact subgroup of $\SL(n,\R)$. Then we have 
\[ 2\delta_{{1,n}}(\G) \leq \DimH(\Lambda^{\mathrm{sym}}_\G) \]
\end{theorem}

The proof is divided into two parts. First, we use  the Hilbert distance on a $\G$-invariant proper convex domain $\Omega\subset\proj(\R^n)$ to establish the equality between $2\delta_{1,n}(\G)$ and the  Hausdorff dimension of $\Lambda_\G\subset \partial\Omega$ for Gromov's ``quasi-distance'' on $\partial \Omega$. Then, we compare this quasi-distance with a Riemannian distance on $\proj(\R^n)\times\proj({\R^n}^*)$.

\subsection{Gromov metric on the boundary}\label{subsec-gromov metric on the boundary}
Let $\G$ be a strongly projectively convex-cocompact subgroup of $\SL(n,\R)$. Let $\Omega$ be a $\G$-invariant Gromov hyperbolic convex domain $\Omega\subset\proj(\R^n)$ and $\cC\subset \Omega$ a closed $\G$-invariant subset of $\Omega$ on which $\G$ acts cocompactly. Let $d_\Omega$ denote the Hilbert distance on $\Omega$.

Recall that Theorem \ref{theo - BoundaryMapsConvexCocompact} states that $\Lambda_\Gamma$ is the intersection of the closure of $\cC$ in $\proj(\R^n)$ with $\partial \Omega$ and that $\Lambda_\G^*$ is the set of hyperplanes tangent to $\Omega$ at a point in $\Lambda_\G$.

Given $x\in \Omega$ and $\xi,\eta\in \partial\Omega$, we define  the Gromov product $(\xi\vert\eta)_x$ by
\[( \xi \vert \eta)_x  =  \lim_{k\tv +\infty} \frac{1}{2}[ d_\Omega(x_k, x) +d_\Omega(y_k, x) -d_\Omega(x_k, y_k)],\]
where $(x_k)$ and $(y_k)$ are sequences in $\Omega$ such that $x_k \tv \xi $ and $y_k\tv \eta$. 

\begin{defi} \label{def - gromov distance}
The \emph{Gromov quasi-distance} between $\xi$ and $\eta$ is the quantity 
\[d_x(\xi, \eta)= e^{-( \xi\vert \eta)_x }.\]
\end{defi}

In general, $d_x$ is not quite a distance. However, Gromov proves that $d_x$ behaves coarsely like some power of a distance. More precisely:
\begin{prop}[Gromov] Let $(X,d)$ be a complete Gromov hyperbolic space, and let $x\in X$. For some sufficiently small $\epsilon$, there exists a constant $C>1$ and a distance $\bar d_{x,\epsilon}$ on $\partial_\infty X$ such that
\[\frac{1}{C} \bar d_{x,\epsilon} \leq d_x^\epsilon \leq C \bar d_{x,\epsilon}~.\]
\end{prop}
We can then define the Hausdorff dimension of $(\partial_\infty X, d_x)$ by\footnote{Alternatively, one could copy the definition of Hausdorff dimension using covering by ``quasi-balls'' for the quasi-distance $d_x$.}
\[\DimH(\partial_\infty X, d_x) = \frac{1}{\epsilon} \DimH(\partial_\infty X, \bar d_{x,\epsilon})~.\]
Proposition \ref{p:GeneralComparisonHDim} implies that this definition is independent of the choices of $\epsilon$ and $d_{x,\epsilon}$.

We can apply here the main result of \cite{Coornaert}:

\begin{theorem}[Coornaert, \cite{Coornaert}] \label{theo - coornaert}
Let $(X,d)$ be a complete Gromov hyperbolic space, and let $\G$ be a discrete group of isometries that acts cocompactly on $X$. Fix $x\in X$, and consider the Gromov quasi-distance $d_x$ on the visual boundary $\partial_\infty X$. Then \[\DimH(\partial_\infty X,d_x)= \delta(\G)~,\] where $\delta(\G)$ is the critical exponent of the action of $\G$ on $(X,d)$.
\end{theorem}
Applying this result to $(X,d)=(\cC,d_\Omega)$, we get that \[\DimH(\Lambda_\G,d_x)=\delta_H(\G)~.\] By Corollary \ref{cor:deltaH=2delta1n}, we conclude:

\begin{prop} \label{prop - dimension Gromov = exposant critique} Let $\G\subset \SL(n,\R)$ be acting convex-cocompactly on a Gromov-hyperbolic convex domain $\Omega\subset\proj(\R^n)$. For any $x\in \Omega$, we have:
\[ \DimH(\Lambda_\G,d_x)=2\delta_{1,n}(\G).\]
\end{prop}


\subsection{Gromov distance VS Euclidean distance}

We keep the same notations as in the previous subsection. We now wish to show that $\DimH(\Lambda_\G,d_x)\leq \DimH(\Lambda_\G^{\mathrm{sym}})$. For $p\in \partial \Omega$, let $p^*=T_p\partial\Omega \in \proj({\R^n}^*)$. The required inequality will easily follow from the following comparison lemma:

\begin{lemma} \label{lemma gromov smaller euclid}
Given $d_\proj$ and $d_\proj^*$ Riemannian distances on $\proj(\R^n)$ and $\proj({\R^n}^*)$, there is a constant $C>0$ such that:
\[ \forall p,q\in\partial \Omega~ d_x(p,q) \leq C \sqrt{d_\proj(p,q) d_\proj^*(p^*,q^*)}\]
\end{lemma}

\begin{proof}
First of all, since $\Omega$ and $\Omega^*$ are proper convex sets, we can assume without loss of generality that $d_\proj$ and $d_\proj^*$ are Euclidean distances in affine charts in which $\Omega$ and $\Omega^*$ are bounded.

Consider sequences $p_n\in[xp)$, $q_n\in[xq)$ that converge to $p$ and $q$ respectively. Consider $p^-$ (resp. $q^-$) the other intersection point between $\partial \Omega$ and the projective line $(xp)$ (resp. $(xq)$). Finally, consider $a_n,b_n\in \partial \Omega$ the endpoints of the geodesic joining $p_n$ and $q_n$ (see Figure \ref{figure gromov distance}).

\begin{figure}[h]
\begin{tikzpicture}[line cap=round,line join=round,>=triangle 45,x=1.0cm,y=1.0cm,scale=0.5]
\clip(-3.5,-8) rectangle (20,8.05699233262419);
\draw [rotate around={32.730897488532186:(7.71,-1.62)}] (7.71,-1.62) ellipse (8.500687972956655cm and 3.0025982104787134cm);
\draw (5.41442766453723,0.45771305753372005)-- (7.26,-2.54);
\draw (13.64,3.6)-- (7.26,-2.54);
\draw (1.8144447318306245,-2.704193078492863)-- (14.96254162624239,2.773947370585196);
\draw [shift={(10.493891698519631,3.969887782684295)}] plot[domain=-0.11703268366256303:0.6049618762811327,variable=\t]({1.*0.681832435596674*cos(\t r)+0.*0.681832435596674*sin(\t r)},{0.*0.681832435596674*cos(\t r)+1.*0.681832435596674*sin(\t r)});
\draw (7.751563986137476,-2.0468817028004826) node[anchor=north west] {$x$};
\draw [dash pattern=on 6pt off 6pt] (7.26,-2.54)-- (2.8001903823373118,-6.8320424847098575);
\draw [dash pattern=on 6pt off 6pt] (7.26,-2.54)-- (8.54690646511876,-4.630287245958173);
\draw (8.913509500211314,-4.572850211656651) node[anchor=north west] {$p^-$};
\draw (3.1037819298421243,-6.694663759095833) node[anchor=north west] {$q^-$};
\draw (4.871959886041442,1.312656413978221) node[anchor=north west] {$p$};
\draw (13.990706203012216,4.59641547549124) node[anchor=north west] {$q$};
\draw (15,3) node[anchor=north west] {$b_n$};
\draw (11.212140843270427,1.2621370438010977) node[anchor=south east] {$q_n$};
\draw (6.6148781571522,-0.5313005974867817) node[anchor=north west] {$p_n$};
\draw (1.8407976754140396,-2.4762963493060313) node[anchor=north west] {$a_n$};
\draw [dash pattern=on 6pt off 6pt] (14.96254162624239,2.773947370585196)-- (16.21776025931955,3.2969328658639094);
\draw [dash pattern=on 6pt off 6pt] (1.8144447318306245,-2.704193078492863)-- (-0.6335899569568746,-3.724164102215152);
\draw (-1.6197791817189124,-2.5520754045717164) node[anchor=north west] {$a'_n$};
\draw (16.466155341691262,4.1670008289856915) node[anchor=north west] {$b'_n$};
\draw (11,5) node[anchor=north west] {$\gamma(p,q)$};
\draw (4.676918376644277,4.653789021223915)-- (19.88073796019969,2.866276775873923);
\draw (12.172636678756389,5.1306491642228735)-- (-1.7450583457772426,-4.492684394297657);
\end{tikzpicture}
\caption{Computing the Gromov product}
\label{figure gromov distance}
\end{figure}
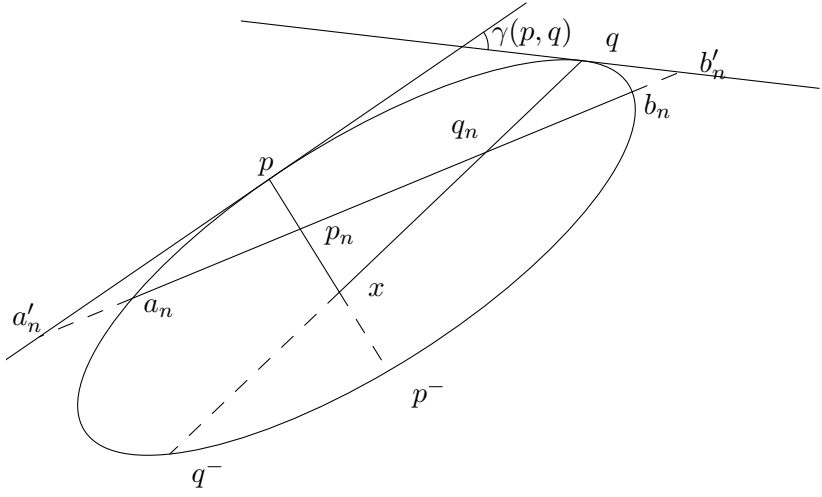

We have that:
\begin{align*}
(p_n\vert q_n)_x &= \frac{1}{2}\left( d_\Omega(x,p_n)+d_\Omega(x,q_n)-d_\Omega(p_n,q_n)\right)\\
&= \frac{1}{4} \Log\left( \frac{p^-p_n\cdot px}{p^-x\cdot pp_n} \cdot\frac{q^-q_n \cdot qx}{q^-x\cdot qq_n}\cdot \frac{a_np_n\cdot b_nq_n}{a_nq_n\cdot b_np_n} \right) \\
&= \frac{1}{4}\Log \left(\frac{a_np_n\cdot b_nq_n}{pp_n\cdot qq_n}  \right)+\underbrace{\frac{1}{4}\Log \frac{1}{a_nq_n\cdot b_np_n}}_{\to -\Log \sqrt{d(p,q)}} + \underbrace{\frac{1}{4}\Log \left(\frac{p^-p_n\cdot px\cdot q^-q_n\cdot qx}{p^-x\cdot q^-x}\right)}_\textrm{bounded}
\end{align*}

This gives us a constant $C_1>0$ such that:
\[ e^{-(p_n\vert q_n)_x} \leq C_1 \sqrt{pq} \left( \frac{pp_n}{a_np_n}\cdot \frac{qq_n}{b_nq_n}\right)^\frac{1}{4} \]

In order to deal with the terms $\frac{pp_n}{a_np_n}$ and $\frac{qq_n}{b_nq_n}$, we consider the affine plane  $\mathcal P_{p,q}$ containing $x,p,q$.  Note that it also contains all the points defined above.

Denote by $a'_n$ (resp. $b'_n$) the intersection of the line $(p_na_n)$ with the tangent space to $\partial \Omega$ at $p$. Note that $\frac{a_np_n}{a'_np_n}\to 1$ as $n\to +\infty$, so that we can work with $\frac{pp_n}{a'_np_n}$ instead of $\frac{pp_n}{a_np_n}$.

Now look at the triangle $a'_np_np$, denote by $\alpha_n$ the angle at $a'_n$, and $\theta(p)$ the angle at $p$. The latter does not depend on $n$ as it is the angle between the line $(xp)$ and the tangent line $T_p\partial\Omega \cap \mathcal P_{p,q}$ (see Figure \ref{figure angles gromov distance}). The law of sines gives us $\frac{pp_n}{a'_np_n}=\frac{\sin \alpha_n}{\sin \theta(p)}$.

We now consider the triangle $b'_nq_nq$, and denote by $\beta_n$ the angle at $b'_n$, and $\p(q)$ the angle at $q$. Just as in the previous case, we get $\frac{qq_n}{b'_nq_n}=\frac{\sin \beta_n}{\sin\p(q)}$.

We now find:
\[ e^{-(p_n\vert q_n)_x} \leq \frac{C_1}{(\sin \theta(p) \sin \p(q))^\frac{1}{4}} \sqrt{pq} \left( \sin \alpha_n \sin \beta_n \right)^\frac{1}{4} \]

Notice that the sequence $(\alpha_n)$ (resp. $(\beta_n)$) has a limit $\alpha(p,q)$ (resp. $\beta(p,q)$) which is the angle at $p$ (resp. at $q$) between the line $(pq)$ and $T_p\partial\Omega \cap \mathcal P_{p,q}$ (resp. $T_q\partial\Omega \cap \mathcal P_{p,q}$).

\begin{figure}[h]
\begin{tikzpicture}[line cap=round,line join=round,>=triangle 45,x=1.0cm,y=1.0cm,scale=0.45]
\clip(2,-10) rectangle (30,6);
\draw (2.722715048128429,-7.643843188948834)-- (15.392017102066404,4.832525331172827);
\draw (10.161374629712993,3.5034276537715505)-- (29.69053566378659,-0.33380596324181266);
\draw (8.49064911030651,-1.9637456860932616)-- (21.027452918105787,1.3683804379402549);
\draw (5.9630864754107265,-4.452817519442409)-- (28.18047073874233,-0.03709726776110722);
\draw (8.49064911030651,-1.9637456860932616)-- (16.163751237331663,-7.965399078642691);
\draw (21.027452918105787,1.3683804379402549)-- (19.722303083277016,-7.7939026041393005);
\draw [shift={(8.49064911030651,-1.9637456860932616)}] plot[domain=3.9193183332717356:5.619412631936143,variable=\t]({1.*0.4707072634791916*cos(\t r)+0.*0.4707072634791916*sin(\t r)},{0.*0.4707072634791916*cos(\t r)+1.*0.4707072634791916*sin(\t r)});
\draw [shift={(8.49064911030651,-1.9637456860932616)}] plot[domain=0.25978144193853264:0.7777256796819427,variable=\t]({1.*0.9765263309181195*cos(\t r)+0.*0.9765263309181195*sin(\t r)},{0.*0.9765263309181195*cos(\t r)+1.*0.9765263309181195*sin(\t r)});
\draw [shift={(13.396812837983548,2.8677048883704206)}] plot[domain=-0.1940157670578886:0.7777256796819425,variable=\t]({1.*0.6715659744278283*cos(\t r)+0.*0.6715659744278283*sin(\t r)},{0.*0.6715659744278283*cos(\t r)+1.*0.6715659744278283*sin(\t r)});
\draw [shift={(21.027452918105787,1.3683804379402549)}] plot[domain=2.9475768865319045:3.401374095528326,variable=\t]({1.*1.5430717203064408*cos(\t r)+0.*1.5430717203064408*sin(\t r)},{0.*1.5430717203064408*cos(\t r)+1.*1.5430717203064408*sin(\t r)});
\draw [shift={(21.027452918105787,1.3683804379402549)}] plot[domain=4.570892808167798:6.089169540121698,variable=\t]({1.*0.6145496508903783*cos(\t r)+0.*0.6145496508903783*sin(\t r)},{0.*0.6145496508903783*cos(\t r)+1.*0.6145496508903783*sin(\t r)});
\draw [shift={(28.18047073874233,-0.03709726776110722)}] plot[domain=2.9475768865319045:3.3377866570401675,variable=\t]({1.*1.753050834530312*cos(\t r)+0.*1.753050834530312*sin(\t r)},{0.*1.753050834530312*cos(\t r)+1.*1.753050834530312*sin(\t r)});
\draw [shift={(5.9630864754107265,-4.4528175194424096)}] plot[domain=0.19619400345037571:0.7777256796819437,variable=\t]({1.*0.7162176818659123*cos(\t r)+0.*0.7162176818659123*sin(\t r)},{0.*0.7162176818659123*cos(\t r)+1.*0.7162176818659123*sin(\t r)});
\draw (14.255852958481443,3.8035464841524838) node[anchor=north west] {$\gamma(p,q)$};
\draw (19.5,1.8742111459893405) node[anchor=north east] {$\beta(p,q)$};
\draw (9.7,0) node[anchor=north west] {$\alpha(p,q)$};
\draw (6.902941614148573,-3.2) node[anchor=north west] {$\alpha_n$};
\draw (8.23203929154985,-2.28457836071788) node[anchor=north west] {$\theta(p)$};
\draw (21.523016065562615,1.102477010724083) node[anchor=north west] {$\varphi(q)$};
\draw (25,0.5022393499622162) node[anchor=north west] {$\beta_n$};
\draw (7.8,-1.2) node[anchor=north west] {$p$};
\draw (21.008526642052445,2.131455857744426) node[anchor=north west] {$q$};
\draw (11.533346425740117,-3.4207425043028423) node[anchor=north west] {$p_n$};
\draw (20.901341345487825,-1.3842218695750796) node[anchor=north west] {$q_n$};
\draw (4.4,-3.677987216057928) node[anchor=north west] {$a'_n$};
\draw (28.361437986385315,0.9524175955336163) node[anchor=north west] {$b'_n$};
\end{tikzpicture}
\caption{}
\label{figure angles gromov distance}
\end{figure}

We thus obtain:

\[ d_x(p,q) \leq \frac{C_1}{(\sin \theta(p) \sin \p(q))^\frac{1}{4}} \sqrt{pq} \left( \alpha(p,q)\beta(p,q) \right)^\frac{1}{4} \]

The function $\theta$ is continuous on the compact set $\partial \Omega$ (because $\Omega$ has $\mathcal C^1$ boundary), and never vanishes (because $x$ is in the interior of $\Omega$), hence is bounded away from $0$. The same goes for $\p$ (notice that $\p(q)=\pi-\theta(q)$), and we can thus find a constant $C_2>0$ such that:
\[ d_x(p,q)\leq C_2 \sqrt{pq} (\alpha(p,q)\beta(p,q))^\frac{1}{4}~.\]

Consider now the exterior angle $\gamma(p,q)$ between the lines $T_p\partial\Omega \cap \mathcal P_{p,q}$ and $T_q\partial\Omega \cap \mathcal P_{p,q}$.\\
Notice that we have $\alpha(p,q)+\beta(p,q)=\gamma(p,q)$. Using the inequality between arithmetic and geometric means, we deduce:
\[ d_x(p,q)\leq C_2 \sqrt{pq} \sqrt{\gamma(p,q)} \]

Finally, there is a constant $C_3$ such that the angle $\gamma(p,q)$ between the lines $p^*\cap \mathcal P_{p,q}$ and $q^* \cap \mathcal P_{p,q}$ is smaller than $C_3 d^*(p^*,q^*)$. This gives the desired inequality.
\end{proof}


%

Let us now conclude the proof of Theorem \ref{th - lower bound}.

\begin{proof}[Proof of Theorem \ref{th - lower bound}]
By Lemma \ref{lemma gromov smaller euclid}, we have
\[d_x(p,q) \leq \frac{C}{\sqrt{2}}\sqrt{d(p,q)^2 + d^*(p^*,q^*)^2}\]
for all $p,q \in \Lambda_\Gamma$. Since $\sqrt{d(p,q)^2 + d^*(p^*,q^*)^2}$ is a Riemannian distance on a neigbourhood of $\Lambda^{\mathrm{sym}}(\G) \subset \proj(\R^n) \times \proj({\R^n}^*)$, we deduce that
\[\DimH(\Lambda_\G, d_x) \leq \DimH(\Lambda^{\mathrm{sym}}(\G))~.\]
Since
\[2\delta_{1,n}(\G) = \DimH(\Lambda_\G, d_x)\]
by Corollary \ref{prop - dimension Gromov = exposant critique}, the theorem follows.
\end{proof}

Let us finally recall the following consequence for every projective Anosov group:
\begin{coro}
Let $\G \subset \SL(n,\R)$ be a projective Anosov group. Then
\[\delta_{1,n}(\G) \leq  \DimH(\Lambda^{\mathrm{sym}}(\G))~.\]
\end{coro}

\begin{proof}
Let $I$ be the representation of $\SL(n,\R)$ into $\SL(\mathrm{Sym}^2({\R^n}^*))$. Then $I(\G)$ is strongly projectively convex cocompact. Applying Theorem \ref{th - lower bound}, we obtain
\[\delta_{1,n}(\G) = 2 \delta_{1, n(n+1)/2}(I(\G)) \leq \DimH(\Lambda^{\mathrm{sym}}(\G))~.\]
\end{proof}

\section{Upper bound} \label{sec - upper bound}

In this section we prove the upper inequality for the Hausdorff dimension of the limit set of a general projective Anosov subgroup: 

\begin{theorem} \label{th - upper bound}
Let $\G\subset \SL(n,\R)$ be a projective Anosov subgroup. Then
$$\DimH(\Lambda_\G) \leq \delta_{1,2}(\G)~.$$
\end{theorem}

Applying this theorem to $\rho(\G)$ where $\rho: \SL(n,\R) \to \SL(\R^n \otimes {\R^n}^*)$ is the tensor product of the standard representation with its dual, we obtain the \textit{a priori} stronger inequality:

\begin{coro} \label{c - upper bound symmetric}
Let $\G\subset \SL(n,\R)$ be a projective Anosov subgroup. Then
$$\DimH(\Lambda^{\mathrm{sym}}_\G) \leq \delta_{1,2}(\G)~.$$
\end{coro} 

\begin{proof}[Proof of Corollary \ref{c - upper bound symmetric} assuming Theorem \ref{th - upper bound}]
Let $\rho: \SL(n,\R) \to \SL(\R^n \otimes {\R^n}^*)$ denote the tensor product of the standard representation with its adjoint. Recall that we have
\[\Lambda_{\rho(\Gamma)}= \Lambda_\Gamma^{\mathrm{sym}}~.\]
(See Proposition \ref{prop - boundary map tensor product} and Example \ref{example - Symmetric boundary map}.)
By Theorem \ref{th - upper bound}, we have
\[\DimH(\Lambda^{\mathrm{sym}}_\G) \leq \delta_{1,2}(\rho(\G))~.\]
By Proposition \ref{prop - Critical exp SW tensor product}, we have
\[\delta_{1,2}(\rho(\G)) = \sup \{\delta_{1,2}(\G), \delta_{n-1,n}(\G)\}~.\] Finally, since $(\mu_{n-1}-\mu_n)(\g) = (\mu_1 - \mu_2)(\g^{-1})$, we have
\[\delta_{1,2}(\G) = \delta_{n-1,n}(\G)~.\]
We conclude that 
\[\DimH(\Lambda_\Gamma^{\mathrm{sym}}) = \DimH(\Lambda_{\rho(\Gamma)}) \geq \delta_{1,2}(\rho(\G)) = \delta_{1,2}(\G)\]
and the corollary follows.
\end{proof}


Let us now turn to the proof of Theorem \ref{th - upper bound}. The main technical tool for the proof is Lemma \ref{lem - quantified distortion of balls by loxodromic element} that quantifies the distortion of balls by proximal elements. The second part of the proof presents a covering of the limit set by images of a fixed ball in order to obtain the upper bound. 


%
%
%
%

\subsection{Distortion of balls by proximal elements}

As in Section \ref{subsec - Critical exponents symmetric space}, we endow $\proj(\R^n)$ with the Riemannian distance $d_\proj$ that lifts to the round metric on $\mathbb{S}^{n-1}$. Recall that an element $g\in\SL(n,\R)$ is \emph{proximal} if it has an attracting fixed point in $\proj(\R^n)$. We denote this attracting point by $L_+(g)$ and the hyperplane complement of its basin of attraction by $H_-(g)$. If $d_\proj(L_+(g), H_-(g)) > r$, then for every $\epsilon>0$ there exists $k \in \N$ such that $g^k$ is $(r,\epsilon)$-proximal (see definition \ref{def - proximality}).

The following lemma states that the contraction of a proximal element $g$ near its attracting fixed point is controlled by $e^{(\mu_2-\mu_1)(g)}$.


\begin{lemme}\label{lem - quantified distortion of balls by loxodromic element}
For every $r>0$, there exists $\epsilon >0$ and a constant $C>0$ such that every $(2r,\epsilon)$-proximal matrix $g\in \SL(n,\R)$ is $Ce^{(\mu_2-\mu_1)(g)}$-Lipschitz on $B(L_+(g),r)$.
\end{lemme}

\begin{proof}
The proof is similar to that of Lemma \ref{lem - ComparisonMuLambdaProximalElement}.

We can find a compact set $M \subset \SL(n,\R)$ (depending only on $r$) that any $(2r,\epsilon)$-proximal element $g$ can be written as $m^{-1}hm$ with $m\in M$ and $h$ proximal satisfying $L_+(h) = [e_1]$ and $H_-(h) = e_1^\perp$. Let $C_1>0$ be such that every $m\in M$ is $C_1$-bilipschitz.

If $\epsilon < \frac{1}{C_1^2}$, then $h$ is contracting at $L_+(h)$, from which we deduce that $h$ acts as a $e^{(\mu_2-\mu_1)(h)}$- contracting linear map in the affine chart $x_1=1$. 

Now, $m(B(L_+(g),r))$ is contained in the domain $T = \{v \mid d_\proj(v, e_1^\perp) \geq \frac{r}{C_1}\}$. There exists a constant $C_2$ such that the Euclidean metric of the affine chart $x_1=1$ and the round metric are $C_2$-bilipschitz on $T$. We deduce that $h$ is $C_2^2e^{(\mu_2-\mu_1)(h)}$-Lipschitz on $T$, hence $g$ is $C_1^2 C_2^2e^{(\mu_2-\mu_1)(h)}$ on $B(L_+(g),r)$. Finally, by Proposition \ref{prop - Benoist cartan projection continuous}, there exists a constant $C_3$ such that 
\[\vert(\mu_2-\mu_1)(g)- (\mu_2-\mu_1)(h)\vert \leq C_3 ~,\] and we conclude that $g$ is $Ce^{(\mu_2-\mu_1)(g)}$-Lipschitz on $B(L_+(g),r)$ for $C=C_1^2 C_2^2e^{C_3}$.
\end{proof}

%
%
%
%
%
%

\subsection{Proof of the upper bound}
In order to bound the Hausdorff dimension of a set from above, it is sufficient to find a cover of this set by sufficiently small balls. We will show in Lemma \ref{lem - covering the limit set} that $\Lambda_\Gamma$ can be covered by translates of a ball by some particular proximal elements.

Let $x,y$ be two distinct points of $\Lambda_\Gamma$. By transversality of the boundary maps, there exists $r>0$ such that, for all $x'\in \xi^{-1}(B(x,r))$ and all $y' \in \xi^{-1}(B(y,r))$, we have
\[d_\proj(\xi(x'), \xi^*(y')) > 6r~.\]

Denote respectively by $U$ and $V$ the preimages of $B(x,r)$ and $B(y,r)$ by $\xi$. Note that $U$ and $V$ are neighbourhoods of $\xi^{-1}(x)$ and $\xi^{-1}(y)$ respectively.

For $\epsilon>0$, define $\G_\epsilon$ to be the set of elements $\g\in \G$ such that 
\begin{itemize}
\item $\g(V^c)\subset U$,
\item $\g$ is $\epsilon$-Lipschitz on $B(x, 5r)$.
\end{itemize}

\begin{prop}
For $\epsilon < \frac{2}{3}$, the set $\Gamma_\epsilon$ is a non-empty semigroup.
\end{prop}

\begin{proof}
Let $\gamma$ and $\gamma'$ be elements in $\G_\epsilon$. Since $U\cap V = \emptyset$, we have $\gamma\gamma'(V^c) \subset U$. 

Note also that $\gamma'_+ \in U$, hence $L_+(\gamma')\in B(x,r)$ and $B(x,5r) \subset B(L_+(\gamma'),6r)$. Since $\gamma'$ is $\epsilon$-Lipschitz on $B(x,5r)$, we deduce that $\gamma' B(x,5r) \subset B(L_+(\gamma'), 6\epsilon r) \subset B(x,5r)$. Since $\gamma$ is $\epsilon$-Lipschitz on $B(x,5r)$, we conclude that $\gamma \gamma'$ is $\epsilon^2$-Lipschitz on $B(x, 5r)$. Hence $\Gamma_\epsilon$ is a semigroup.

Moreover, by Corollary 8.2.G of \cite{Gromov}, the set of pairs $(\g_+,\g_-)$ of hyperbolic elements $\g\in\G$ is dense in $\partial_\infty \G\times\partial_\infty\G$. Hence there exists $\gamma \in \Gamma$ such that $\gamma_+ \in U$ and $\gamma_-\in V$. By definition of $U$ and $V$, we have $d_\proj(L_+(\gamma),H_-(\gamma)) > 6r$. Hence, after replacing $\gamma$ by a large enough power, we can assume that $\gamma$ is $(6r,\epsilon)$-proximal for an arbitrarily small $\epsilon$. Since $B(x,5r)\subset B(L_+(\gamma), 6r)$, we conclude that $\gamma \in \Gamma_\epsilon$. Hence $\Gamma_\epsilon$ is non-empty.
\end{proof}

\begin{lemme} \label{lem - covering the limit set}
For any $0<\epsilon< \frac{2}{3}$ we have
\[\Lambda_\Gamma\cap B(x,r) \subset \cup_{\g \in \G_\epsilon} \g \cdot B(x,r)~.\]
\end{lemme}
 
\begin{proof}
Set $O^\epsilon =  \cup_{\g \in \G_\epsilon} \g \cdot B(x,r)$. 
By definition, it is a $\G_\epsilon$-invariant open subset of $B(x,r)$. Moreover, $\Lambda_\Gamma \cap O^\epsilon$ is non-empty since it contains $\G_\epsilon \cdot x$. The set $C^\epsilon = \Lambda_\G \cap (B(x,r)\setminus O^\epsilon)$ is closed in $B(x,r)$ and $\G_\epsilon$-invariant. We want to prove that $C^\epsilon$ is empty. Assume by contradiction that it is not the case, and pick $c \in  \xi^{-1}(C^\epsilon)$. 

Let $\g \in \G$ be a proximal element such that $\g_- \in V$ and $\g_+ \in U$. Then for $k$ large enough, $\g^k$ belongs to $\G_\epsilon$. Since $c\neq \g_-$, $\g^k c$ converges to $\g_+$ as $k$ goes to $+\infty$. Since $C^\epsilon$ is $\G_\epsilon$-invariant and closed, we obtain that $L_+(\g)\in C^\epsilon$. By density of the pairs $(\g_+,\g_-)$ in $\partial_\infty \G\times\partial_\infty\G$, we conclude that $C^\epsilon = \Lambda_\Gamma\cap B(x,r)$, contradicting the fact that $\Lambda_\Gamma \cap O^\epsilon$ is non-empty.
\end{proof}

We now have all the tools to prove Theorem \ref{th - upper bound}.

\begin{proof}[Proof of Theorem \ref{th - upper bound}]
By compactness of $\Lambda_\Gamma$ it is sufficient to prove that, for all $x\in \Lambda_\Gamma$, there exists $r>0$ such that the Hausdorff dimension of $\Lambda_\Gamma \cap B(x,r)$ is less than $ \delta_{1,2}(\G)$.

Given such and $x$, let us fix $y, r, U$, and $V$ as before. Let $C>0$ be given by  Lemma \ref{lem - quantified distortion of balls by loxodromic element}.

By Lemma \ref{lem - covering the limit set}, we have $\Lambda_\Gamma\cap B(x,r) \subset  \cup_{\g \in \G_\epsilon} \g \cdot B(x,r)$. By definition of $\G_\epsilon$, this gives a covering of $\Lambda_\Gamma\cap B(x,r)$ by balls of radius less than $\epsilon r$.

Note that every $\gamma \in \Gamma_\epsilon$ is $\epsilon$-Lipschitz on $B(L_+(\gamma), 4r) \subset B(x,5r)$, hence it is $(4r,\epsilon)$-proximal. By Lemma \ref{lem - quantified distortion of balls by loxodromic element}, it is thus $Ce^{(\mu_2-\mu_1)(\gamma)}$-Lipschitz on $B(x,r) \subset B(L_+(\g),2r)$, for some constant $C$.

Let $s>0$. By definition of the $s$-dimensional Hausdorff measure,  we have: 

\begin{eqnarray*}
H^{s}_{r\epsilon} (\Lambda_\Gamma \cap B(x,r)) &\leq & \sum_{\g\in \G_\epsilon} \mathrm{diam}(\g \cdot B(x,r)) ^s\\
& \leq & C^s  \sum_{\g\in \G_\epsilon}  e^{s(\mu_2- \mu_1)(\g)}\\
&\leq & C^s \sum_{\g\in \G}  e^{-s(\mu_1- \mu_2) (\g)}~. \\			
\end{eqnarray*}
Since $\epsilon$ can be taken arbitrarily small, we obtain : 
$$H^{s} (\Lambda_\G \cap B)\leq C^s \sum_{\g\in \G}  e^{-s(\mu_1- \mu_2) (\g)}~.$$
Therefore, for all $s>\delta_{1,2}$, 
$H^{s} (\Lambda_\G \cap B(x,r))<+\infty$ which in turn implies that $\DimH(\Lambda_\G\cap B(x,r)) \leq \delta_{1,2}$.

\end{proof}


\begin{thebibliography}{GJT12}

\bibitem[AMS95]{abels1995semigroups} 
Herbert Abels, Gregory  Margulis, Gregory  Soifer.
\newblock Semigroups containing proximal linear maps
\newblock {\em Israel J. Math.} (1995) 91, 1-30. 


\bibitem[Ben97]{benoist1997proprietes}
Yves Benoist.
\newblock Propri{\'e}t{\'e}s asymptotiques des groupes lin{\'e}aires.
\newblock {\em Geometric \& Functional Analysis GAFA}, 7(1):1--47, 1997.



\bibitem[Ben01]{Benoist2001convexe}
Yves Benoist.
\newblock Convexes divisibles.
\newblock{\em Comptes Rendus de l'Academie des Sciences} 332(5), 387-390, 2001. 

\bibitem[CTT17]{CCT17}
Brian Collier, Nicolas Tholozan, Jerremy Toulisse. 
\newblock{The geometry of maximal representations of surface groups into SO(2,n).}
\newblock {\em Duke Mathematical Journal} 168 (15), 2019, p. 2873-2949


\bibitem[Coo93]{Coornaert}
Michel Coornaert.
\newblock Mesures de Patterson-Sullivan sur le bord d'un espace hyperbolique au sens de Gromov.
\newblock{\em Pacific J. Math} 159(2), 241-270, 1993.

\bibitem[CDP90]{CoornaertDelzantPapadopoulos}
Michel Coornaert, Thomas Delzant, Athanase Papadopoulos.
\newblock Geometrie et th\'eorie des groupes : les groupes hyperboliques de Gromov.
\newblock {\em Springer lecture notes in Mathematics} 1441, Berlin, 1990.

\bibitem[CK02]{CK}
Michel Coornaert, Gerhard Knieper.
\newblock Growth of conjugacy classes in Gromov hyperbolic groups
\newblock {\em Geom. Funct. Anal.} 12 (3), 464-478, 2002. 

\bibitem[Cra09]{crampon2009entropies}
Michael Crampon.
\newblock Entropies of strictly convex projective manifolds.
\newblock {\em Journal of Modern Dynamics}, 3(4):511--547, 2009.

\bibitem[Cra11]{crampon2011dynamics}
Micka{\"e}l Crampon.
\newblock { Dynamics and entropies of Hilbert metrics}.
\newblock PhD thesis, Universit{\'e} de Strasbourg; Ruhr-Universit{\"a}t
  Bochum, 2011.
  
\bibitem[CM14]{crampon2014finitude}
  Micka{\"e}l Crampon, Ludovic Marquis.
  \newblock Finitude géométrique en géométrie de Hilbert (with an appendix by C. Vernicos)
  \newblock {\em Ann. Inst. Fourier} 64, 2299--2377, 2014.
  
  \bibitem[DOP00]{DOP}
Francoise Dal’bo,  Jean-Pierre Otal, Marc Peigne .
\newblock {Séries de Poincaré des groupes géométriquement finis}.
\newblock{Isral Jour. Math.} 118, 109--124, 2000.


\bibitem[DGK17]{DGK}
Jeffrey Danciger, Fanny Kassel, Francois Guéritaud.
\newblock {Convex cocompactness in real projective space.}
\newblock {\em arXiv:1704.08711} 

\bibitem[DGK18]{DGKSOpq}
Jeffrey Danciger, Fanny Kassel, Francois Guéritaud.
\newblock {Convex cocompactness in pseudo-riemannian hyperbolic spaces.}
\newblock {\em Geom Dedicata } 192:87–126, 2018. 


\bibitem[DK19]{DeyKapovich}
Subhadip Dey, Michael Kapovich.
\newblock{Patterson-Sullivan theory for Anosov subgroups.}
\newblock{\em Transaction of the American Mathematical Society}, Volume 375, Number 12, December 2022, Pages 8687–8737.



\bibitem[Ebe96]{eberlein1996geometry}
Patrick Eberlein.
\newblock{Geometry of Nonpositively Curved Manifolds.}
\newblock{\em Chicago lectures in Mathematics} 1996


\bibitem[GM18]{GM18}
Olivier Glorieux, Daniel Monclair. 
\newblock{Critical exponent and Hausdorff dimension in pseudo-Riemannian hyperbolic geometry.}
\newblock {\em International Mathematics Research Notices}, Volume 2021, Issue 18, September 2021, Pages 13661–13729.



\bibitem[Gro87]{Gromov}
Mikhael Gromov. 
\newblock{Hyperbolic groups\em Essays in group Theory}
\newblock{\em MSRI publications}  8, 1987. 


\bibitem[GGKW17]{gueritaud20017anosov}
Fran{\c{c}}ois Guéritaud, Olivier Guichard, Fanny Kassel, Anna Wienhard.
\newblock{Anosov representations and proper actions.}
\newblock{\em Geometry and Topology} 21(1), 485-584.


\bibitem[Gui17]{Guichard}
Olivier Guichard.
\newblock{Groupes convexes-cocompacts en rang supérieur, d'après Labourie, Kapovich--Leeb--Porti...}
\newblock{\em Séminaire Bourbaki} 1138, 2017. 

\bibitem[GW12]{guichard2012anosov}
Olivier Guichard, Anna Wienhard.
\newblock{Anosov representations: domains of discontinuity and applications.}
\newblock{\em Invent. math.} (2012) 190(2), 357-438.



\bibitem[Hel78]{helgason1978differential}
Sigurdur Helgason.
\newblock {\em Differential geometry and symmetric spaces}, volume~12.
\newblock Academic press, 1978.



\bibitem[KLP17]{kapovich2017anosov}
Michael Kapovich, Bernhard Leeb, Joan Porti. 
\newblock{Anosov subgroups: dynamical and geometric characterizations}
\newblock{\em European J. of Math} 3(4), 808-898



\bibitem[KLP18]{kapovich2018anosov}
Michael Kapovich, Bernhard Leeb, Joan Porti. 
\newblock{A Morse lemma for quasigeodesics in symmetric spaces and euclidean buildings}
\newblock{\em Geometry and Topology} 22(7), 3827-3923, 2018. 


\bibitem[Lab06]{labourie2006anosov}
Fran{\c{c}}ois Labourie.
\newblock Anosov flows, surface groups and curves in projective space.
\newblock {\em Inventiones Mathematicae}, 165(1):51--114, 2006.

\bibitem[Lin04]{link2005measures}
Gabriele Link
\newblock Measures on the geometric limit set in higher rank symmetric spaces.
\newblock {\em Actes du Séminaire de TSG, Grenoble}, 22,59--69, 2004..

\bibitem[Mes07]{mess2007lorentz}
Geoffrey Mess.
\newblock Lorentz spacetimes of constant curvature.
\newblock {\em Geometriae Dedicata}, 126(1):3--45, 2007.


\bibitem[Pan89]{pansu1989dimension}
Pierre Pansu.
\newblock Dimension conforme et sphere à l'infini des vari{\'e}t{\'e}sa courbure
  n{\'e}gative.
\newblock {\em Ann. Acad. Sci. Fenn. Ser. AI Math}, 14(2):177--212, 1989.


\bibitem[Pau97]{paulin1997critical}
Fr{\'e}d{\'e}ric Paulin.
\newblock On the critical exponent of a discrete group of hyperbolic
  isometries.
\newblock {\em Differential Geom. Appl.}, 7(3):231--236,1997.

\bibitem[Pei13]{peigne2013autour}
Marc Peign{\'e}.
\newblock Autour de l'exposant de poincar{\'e} d'un groupe kleinien.
\newblock {\em Monogr.  Enseign. Math.}, 43, 2013.


\bibitem[PS17]{potrie2014}
Raphael Potrie and Andres Sambarino. 
\newblock Eigenvalues and Entropy of a Hitchin representation.
\newblock {\em Inventiones Mathematicae}, 209(3):885--925, 2017


\bibitem[PSW19]{PSW19b}
Beatrice Pozzetti, Andres Sambarino and Anna Wienhard
\newblock Anosov representations with Lipschitz limit set.
\newblock  {\em arXiv:1910.06627}

\bibitem[PSW21]{PSW19}
Beatrice Pozzetti, Andres Sambarino and Anna Wienhard
\newblock Conformality for a robust class of non-conformal attractors.
\newblock {\em Journal für die reine und angewandte Mathematik} (Crelles Journal), vol. 2021, no. 774, 2021, pp. 1-51 




\bibitem[Qui02a]{quint2002divergence}
J-F Quint
\newblock Divergence exponentielle des sous-groupes discrets en rang supérieur.
\newblock {\em Commentarii mathematici Helvetici}, 77: 563--608 2002.





\bibitem[Qui02b]{quint2002mesures}
J-F Quint
\newblock Mesures de Patterson-Sullivan en rang supérieur.
\newblock {\em Geometric and functional analysis}, 12: 776--809, 2002.



\bibitem[Rob03]{roblin2003ergodicite}
Thomas Roblin.
\newblock Ergodicit{\'e} et {\'e}quidistribution en courbure n{\'e}gative.
\newblock {\em Mem. Soc. Math.  Fr.},
  (95):1--96, 2003.
  
  
\bibitem[Sam14]{sambarino2014quantitative}
Andres Sambarino.
\newblock Quantitative properties of convex representations
\newblock {\em Commentarii Mathematici Helvetici}, (89): 443--448, (2014)

\bibitem[Sul79]{sullivan1979density}
Dennis Sullivan.
\newblock The density at infinity of a discrete group of hyperbolic motions.
\newblock {\em Publ. Math. Inst. Hautes \'Etudes Sci.}, 50:171--202,
  1979.
  
\bibitem[Zim17]{zimmer}
Andrew Zimmer.
\newblock Projective Anosov representations, convex cocompact actions, and rigidity.
\newblock To appear in {\em Journal of Differential Geometry}

\end{thebibliography}
\end{document}